\documentclass[12pt]{amsart}
\usepackage[utf8]{inputenc}
\usepackage{amsthm}
\usepackage{amsmath}
\usepackage{stmaryrd}
\usepackage{amssymb} 
\usepackage{comment}
\usepackage{graphicx}
\usepackage{mathrsfs}
\usepackage{fullpage}
\usepackage{mathtools}
\usepackage{colonequals}
\usepackage{hyperref}
\usepackage[all]{xy}
\usepackage[usenames,dvipsnames]{xcolor}
\usepackage[ left = 1.2in, right = 1.2in, 
             top = 1.2in, bottom = 1.2in ]{geometry}
\usepackage{tikz-cd}

\usepackage{tikz}

\numberwithin{equation}{section} 
\numberwithin{figure}{section} 

\makeatletter
\let\c@equation\c@figure
\makeatother

\newcommand{\C}{\mathbb{C}}
\newcommand{\F}{\mathbb{F}}

\newcommand{\N}{\mathbb{N}}
\newcommand{\PP}{\mathbb{P}}
\newcommand{\Q}{\mathbb{Q}}
\newcommand{\R}{\mathbb{R}}
\newcommand{\RR}{\mathbb{R}}

\newcommand{\Z}{\mathbb{Z}}

\newcommand{\cF}{\mathcal{F}}

\newcommand{\cK}{\mathcal{K}}
\newcommand{\cL}{\mathcal{L}}

\newcommand{\cW}{\mathcal{W}}

\DeclareMathOperator{\Des}{Des}
\DeclareMathOperator{\des}{des}
\DeclareMathOperator{\desc}{desc}

\DeclareMathOperator{\Span}{Span}

\DeclareMathOperator{\OI}{OI}
\DeclareMathOperator{\ann}{ann}
\DeclareMathOperator{\Supp}{Supp}
\DeclareMathOperator{\PG}{PG}
\DeclareMathOperator{\mult}{mult}
\DeclareMathOperator{\wt}{wt}
\DeclareMathOperator{\topp}{top}
\DeclareMathOperator{\rk}{rk}

\newtheorem{theorem}[equation]{Theorem}
\newtheorem{corollary}[equation]{Corollary}
\newtheorem{lemma}[equation]{Lemma}
\newtheorem{prop}[equation]{Proposition}

\theoremstyle{definition}
\newtheorem{rem}[equation]{Remark}
\newtheorem{example}[equation]{Example}

\theoremstyle{definition}
\newtheorem{definition}[equation]{Definition}

\newcommand*\ps[1]{{\llbracket#1\rrbracket}}
\newcommand\nc{\newcommand}
\nc\on{\operatorname}

\title{Matroidal Mixed Eulerian Numbers}
\author{Eric Katz and Max Kutler}
\date{}

\begin{document}

\maketitle

\begin{abstract}
We make a systematic study of matroidal mixed Eulerian numbers which are certain intersection numbers in the matroid Chow ring generalizing the mixed Eulerian numbers introduced by Postnikov. These numbers are shown to be valuative and obey a log-concavity relation. We  establish recursion formulas and use them to relate matroidal mixed 
Eulerian numbers to the characteristic and Tutte polynomials, reproving results of Huh--Katz and Berget--Spink--Tseng. Generalizing Postnikov, we show that these numbers are equal to certain weighted counts of binary trees. Lastly, we study these numbers for perfect matroid designs, proving that they generalize the remixed Eulerian numbers of Nadeau--Tewari.
\end{abstract}

\section{Introduction}

Eulerian numbers, which count permutations with a certain number of descents, are a classical part of algebraic combinatorics. Postnikov, in his study of the volumes of permutohedra, introduced mixed Eulerian numbers which are mixed volumes of hypersimplexes. 
Inspired by this work, Berget--Spink--Tseng \cite{BST} defined hypersimplex classes $\gamma_k$ (perhaps motivated by the observation \cite[Remark~3.6]{Huh:tropicalgeometry}) in the matroidal Chow ring $A^*(M)$ and related their intersection numbers to $T_M(1,y)$, a particular specialization of the Tutte polynomial of a rank $r+1$ matroid $M$ on the set $E=\{0,1,\dots,n\}$, satisfying
\[T_M(1,y)=\sum_{S\subseteq E:\rk(S)=r+1} (y-1)^{|S|-r-1}.\]
This is a somewhat surprising result: $T_M(1,y)$ is sensitive to the size of flats; the matroid Chow ring vanishes for matroids with loops and otherwise depends only on the simplification of the matroid. In this paper, we make a systematic study of the intersection numbers of hypersimplex classes, which we dub {\em the matroidal mixed Eulerian numbers}, in hopes of getting a better sense of the information contained in them. These numbers arise as degrees in the matroidal Chow ring  for a rank $r+1$ matroid $M$ on $E$:
\[A_{c_1,\dots,c_n}(M)=\deg_M(\gamma_1^{c_1}\dots\gamma_n^{c_n})\]
for nonnegative integers $c_1,\dots,c_n$ satisfying $c_1+c_2+\dots+c_n=r$.
They specialize to the usual mixed Eulerian numbers in the case where $M$ is the Boolean matroid $U_{n+1,n+1}$.
By expressing the hypersimplex classes in the matroid Chow ring according to Lemma~\ref{l:weights}, we see that the hypersimplex classes are a sum of flats weighted by a rational number depending on their size:
\[\gamma_k=\sum_{S\subset E} \OI(S,T)x_S=\sum_{S\subset E} \mult_{E}(|S|,k) x_S,\]
where $\OI$ is an integer called the {\em over-intersection} and $\mult_{E}$ is a particular rational number.
For that reason, matroidal mixed Eulerian numbers are sensitive to more than just the lattice of flats.

The matroidal mixed Eulerian numbers obey many recursion relations, allowing us to get a handle on some of their values. As their combinatorics are quite involved, we find them most accessible in the contiguous or flatly contiguous case, that is, when the set $\{i\mid c_i\neq 0\}$ involves a range of consecutive integers or a range of consecutive sizes of flats. In these cases, the matroidal mixed Eulerian numbers satisfy an analogue of the classical Eulerian recursion $A(n,k)=(n-k+1)A(n-1,k-1)+kA(n-1,k)$ and a certain deletion/recursion relation. These relations immediately yield expressions for the characteristic and Tutte polynomial, reproducing results of Huh--Katz \cite{HuhKatz} and Berget--Spink--Tseng \cite{BST}:
\begin{align*}
\deg_M(\gamma_1^k\gamma_n^{r-k})&=\mu^k(M)\\
C_{v}(M,y)&=T_M(1,y)C_{v}(U_{r+1,r+1},y)
\end{align*}
where $\mu^k$ is a coefficient of the reduced characteristic polynomial, $C_{v}$ is a polynomial built out of matroidal mixed Eulerian numbers, and $T_M(x,y)$ is the Tutte polynomial of $M$.

Postnikov gave a description of the mixed Eulerian numbers as a sum indexed by certain decorated binary trees. We give the analogous description for matroidal mixed Eulerian numbers, which immediately follows from a monomial expansion in the matroid Chow ring. We consider some cases where these trees are particularly explicit and can be related to counts of flags of flats.

Degrees in the matroid Chow ring can be computed by equivariant localization by virtue of Berget--Eur--Spink--Tseng's equivariant lift of the Bergman class \cite{BEST}. This allows us to observe that that matroidal mixed Eulerian numbers are valuative over matroids and that they have an expression in terms of counts of permutations (Theorem~\ref{t:descentcounts}):
  \[\deg_M(\lambda_1^{d_1}\dots\lambda_{n}^{d_n})=\sum_{w}(-1)^{n-r+\des(w)}
  \]
where $\lambda$'s are certain classes in $A^*(M)$ related to $\gamma_k$'s and the sum is over permutations satisfying a certain descent condition.
This hints at connections between permutation statistics and matroids.

A particular case where the analogies between the matroidal and usual mixed Eulerian numbers are especially clear is that of perfect matroid designs, i.e.,~matroids for which there are integers 
\[1=n_1<n_2<\dots<n_r\]
such that every flat of rank $i$ contains exactly $n_i$ elements. These include uniform matroids, projective geometries, and certain sporadic examples. In this case, the matroidal mixed Eulerian numbers involving only the classes $\gamma_{n_i}$ are of particular interest. These numbers obey a recursion coming from a relation in the matroid Chow ring between $\gamma_{n_i}^2$, $\gamma_{n_i}\gamma_{n_{i+1}}$, and $\gamma_{n_{i-1}}\gamma_{n_i}$. 
We write
\[A_{(c_1,\dots,c_r)_n}(M)=\deg_M(\gamma_{n_1}^{c_1}\dots\gamma_{n_r}^{c_r})\]
when $M$ is a perfect matroid design on $E$.
The constant
\[V_M=\left(\prod_{i=1}^r  N_i 
 \frac{n_{i+1}-n_i}{n_{i+1}}\right)\]
appears when computing these numbers, where $N_i$ is the number of rank $i$ flats in a given $i+1$ flat (which itself is expressible in terms of the $n_i$'s).
We verify that (up to a power of $q$) the remixed Eulerian numbers of Nadeau--Tewari \cite{NT:remixed} are the matroidal mixed Eulerian numbers of the projective geometry $\PG(r,q)$, when $q$ is a prime power.

Postnikov observed that the mixed Eulerian numbers obey the following
properties among many others:

\begin{enumerate}
\item The numbers $A_{c_1,\dots,c_n}$ are positive integers defined
for $c_1,\dots,c_n\geq 0$ such that $c_1+\cdots+c_n=n$.

\item For $1\leq k\leq n$, the number $A_{0^{k-1},n,0^{n-k}}$  
is the usual Eulerian number $A(n,k-1)$, equal to the number of permutations of $\{1,\ldots,n\}$ with exactly $k-1$ descents.
Here and below $0^l$ denotes the sequence of $l$ zeros.

\item We have $\sum 
\frac{1}{c_1!\cdots c_n!}\,A_{c_1,\dots,c_n} = (n+1)^{n-1}$,
where the sum is over $c_1,\dots,c_n\geq 0$ with $c_1+\cdots + c_n = n$.

\item We have $A_{k,0,\dots,0,n-k} = \binom{n}{k}$.

\item We have $A_{1,\dots,1}=n!$.

\item We have $A_{c_1,\dots,c_n} = 1^{c_1} 2^{c_2} \cdots n^{c_n}$
if $c_1+\cdots + c_i \geq i$ for $i=1,\dots,n-1$,
and $c_1+\cdots + c_n = n$.
\end{enumerate}

We establish the following analogues, some of which are very straightforward in our setting:
\begin{enumerate}
\item \label{i:nonneg} The numbers $A_{c_1,\dots,c_n}(M)$ are nonnegative integers, defined
for $c_1,\dots,c_n\geq 0$ such that $c_1+\cdots+c_n=r$.

\item \label{i:eulerianrecurrence} The flatly contiguous matroidal mixed Eulerian numbers obey an analogue of the Eulerian recurrence (Proposition \ref{p:eulerianrelation}).

\item \label{i:pvolume} We have $\sum 
\frac{r!}{c_1!\cdots c_n!}\,A_{c_1,\dots,c_n} = \on{PVol}(M)$,
where the sum is over $c_1,\dots,c_n\geq 0$ with $c_1+\cdots + c_n = r$
and $\on{PVol}(M)$ is the permutohedral volume of $M$ (Lemma \ref{l:pvol}).

\item \label{i:charpoly} We have $A_{k,0,\dots,0,n-k}(M) = \mu^k(M)$ (Proposition \ref{p:charpoly}).

\item \label{i:tutte} We have $A_{1,\dots,1}(M)=r!T_M(1,0)$ (Corollary \ref{c:tutte}).

\item \label{i:perfectlylopsided} For a perfect matroid design $M$, we have $A_{(c_1,\dots,c_r)_n}(M) = V_Mn_1^{c_1} n_2^{c_2} \cdots n_r^{c_r}$
if $c_1+\cdots + c_i \geq i$ for $i=1,\dots,r-1$,
and $c_1+\cdots + c_n = r$ (Lemma \ref{l:lopsided}).
\end{enumerate}

Here, \eqref{i:nonneg}, a consequence of the nefness (i.e.~convexity) of the hypersimplex classes by Theorem~\cite[Theorem~8.9]{AHK}, was noted in \cite{BST}. Item~\eqref{i:charpoly} was proven in \cite{HuhKatz} while \eqref{i:tutte} is a special case of \cite[Corollary~1.6]{BST}. Item~\eqref{i:pvolume} is immediate from definitions. Items~\eqref{i:eulerianrecurrence} and \eqref{i:perfectlylopsided} appear to be new.

Our work is closely related to that of Horiguchi \cite{Horiguchi} who studied the connection between mixed Eulerian numbers in various Coxeter types and the Petersen Schubert Calculus.

We owe particular thanks to the innovative work of and conversations with Andrew Berget, Christopher Eur, Hunter Spink, and Dennis Tseng. This work arose from the authors' ``Faculty 15'' research meeting held at the Traditions at Scott Dining Hall.

In section~\ref{s:matroids}, we review matroid Chow rings and hypersimplex classes, giving a new representative for them. The relations between these classes coming from evaluation and deletion/contraction, described in section~\ref{s:relations}, are employed to give new proofs of expressions for the characteristic and Tutte polynomials in section~\ref{s:tutte}. We describe the tree expansion in section~\ref{s:trees}. Valuativity and formulas in terms of permutations are given by means of localization in section~\ref{s:localization}. Section~\ref{s:perfect} studies perfect matroid designs.

\section{Matroid Chow Rings and Hypersimplex Classes} \label{s:matroids}

\subsection{Matroids and Matroid Chow Rings}

We begin by reviewing some notions of matroids following \cite{AHK}.
Let $E$ denote the set $\{0,1,\dots,n\}$. Let $S_{E}$ denote the group of bijections from $E$ to itself.
Usually, we will take $M$ to be a rank $r+1$ matroid on $E$ given by a rank function $\rk\colon 2^{E}\to\Z_{\geq 0}$. 
We will denote the  {\em Boolean matroid} on $E$ by $U_{n+1,n+1}$; it is characterized by $\rk(I)=|I|$, that is, every subset of $E$ is a flat.
For $I\subseteq E$, let $\overline{I}$ denote the closure of $I$ with respect to $M$.
All matroids will be loopless unless otherwise noted. For a flat $F$, let $M_F$ denote the contraction of $M$ at $F$, i.e.,~the matroid whose underlying lattice of flats is the interval $[F,\hat{1}]$. Let $M^F$ denote the restriction to $F$, which has lattice $[\hat{0},F]$

In $\R^{E}$, let $e_0,\dots,e_n$ denote the standard unit basis vectors. Write $\mathbf{1}=e_0+\dots+e_n$.
Set $N_{\R}=\R^{E}/\R\mathbf{1}$ where we will conflate the $e_i$'s with their images in $N_{\R}$. For a subset $S\subset E$, write
\[e_S=\sum_{i\in S} e_i \in N_{\R}.\]
For a chain of subsets 
\[\mathcal{S}=\{\varnothing\subsetneq S_1\subsetneq \dots\subsetneq S_k\subsetneq E\},\]
write $\sigma_{\mathcal{S}}\subset N_{\R}$ for the cone
\[\sigma_{\mathcal{S}}=\Span_{\geq 0}(e_{S_1},\dots,e_{S_k}).\]
The permutohedral fan $\Delta_{E}$ is the fan in $N_{\R}$ whose cones are $\sigma_{\mathcal{S}}$ as $\mathcal{S}$ ranges over all chains of subsets.
Attached to $\Delta_{E}$ is the permutohedral toric variety $X(\Delta_{E})$.
For $w\in S_{E}$, let $\sigma_w$ be the cone in $\Delta_{E}$ given by
\[\Span_{\geq 0}(e_{S_1},\dots,e_{S_n}),\]
where $S_1\subsetneq \dots\subsetneq S_n$ is the chain of subsets
\[\{w(0)\}\subsetneq \{w(0),w(1)\}\subsetneq\dots\subsetneq \{w(0),w(1),\dots,w(n)\}.\]

The matroid Chow ring \cite{FY:chow}, $A^*(M)=\Z[x_F]/(I+J)$, is generated by $x_F$ for proper non-empty flats $F$ with relations
\begin{align*}
I&=\left\langle x_{F_1}x_{F_2}\mid F_1,F_2 \text{ are not comparable}\right\rangle\\
J&=\left\langle \sum_{F\ni i} x_F- \sum_{F\ni j} x_F\mid i,j\in E\right\rangle.
\end{align*}
There is a natural homomorphism 
\[A^*(U_{n+1,n+1})\to A^*(M)\]
given by
\[x_S\mapsto 
\begin{cases}
x_S &\text{if $S$ is a flat of $M$}\\
0 &\text{else}.
\end{cases}\]
We will sometimes write $x_S$ for the image of $x_S$ under this homomorphism even when $S$ is not a flat of $M$. For a chain of subsets 
\[\mathcal{S}=\{\varnothing\subsetneq S_1\subsetneq\dots\subsetneq S_k\subsetneq E\},\]
we will write $x_{\mathcal{S}}\coloneqq x_{S_1}\dots x_{S_k}$.
By \cite{AHK}, $A^*(M)$ is a Poincar\'{e} duality ring of dimension $r$ equipped with a degree map
\[\deg_M\colon A^r(M)\to \Z,\]
characterized by the property that for any full flag of flats
\[\{\varnothing \subsetneq F_1\subsetneq F_2\subsetneq \dots\subsetneq F_r\subsetneq E\},\]
we have $\deg(x_{F_1}x_{F_2}\dots x_{F_r})=1.$

Write $\theta_i\colon A^*(M\setminus i)\to A^*(M)$ for the pullback defined in \cite[Section~3]{BHMPW:semismall} as
\[\theta_i(x_F)=x_F+x_{F\cup i}.\]
If $i$ is not a coloop of $M$, then $\theta_i$ commutes with the degree map in the sense
\[\deg_M\circ \theta_i=\deg_{M\setminus i}.\]
If $i$ is a coloop, for dimensional reasons, $\deg_M\circ\theta_i=0$. Henceforth, we will write $A^*(M)$ for $A^*(M)\otimes\R$.

\subsection{Hypersimplex Classes}

Let $S_{E}$ act on the standard unit basis vectors in $\RR^{E}$. The standard $k$-hypersimplex, $\Delta(n+1,k)$, is the convex hull of all vectors in the orbit of $e_{\{0,\dots,k-1\}}$ under this action. It lies in the hyperplane $x_0+x_1+\dots+x_n=k$. 
Attached to it is a class 
\[\gamma_{n+1-k}\in A^1(X(\Delta_{E}))\]
in the Chow cohomology ring of the permutohedral toric variety. It arises as the non-equivariant restriction of the support function to any translate $\Delta(n+1,k)-w$ for a vector $w$ with $w_0+\dots+w_n=k$. By \cite[Section~6.1]{CoxLittleSchenck}, this corresponds to the class in the equivariant Chow ring $A_T^1(\Delta_{E})$ given by the support function on $N$,
\[\varphi(u)=\min_{v\in \Delta(n,k)-w}(u\cdot v)\\
=\min_{|T|=k}(u\cdot e_T)-u\cdot w.
\]
The corresponding non-equivariant class is obtained as
\[-\sum_S \varphi(e_S)x_S.\]
We use the convention that $\gamma_k=0$ for $k\leq 0$ or $k\geq n+1$.

\begin{definition}
For a vector $(c_1,\dots,c_n)$ of nonnegative integers with $c_1+\dots+c_n=r$, we define {\em the matroidal mixed Eulerian number}
\[A_{c_1,\dots,c_n}(M)=\deg_M(\gamma_1^{c_1}\dots\gamma_n^{c_n}).\]
\end{definition}

Because the $\gamma_k$'s are given by convex functions and hence are nef, $A_{c_1,\dots,c_n}(M)$ is always nonnegative. For $M=U_{n+1,n+1}$, these specialize to Postnikov's mixed Eulerian numbers. Indeed, those numbers are described as mixed volumes of hypersimplexes. The matroidal mixed Eulerian numbers specialize to intersection numbers in $A^*(\Delta_{E})=A^*(X(\Delta_{E}))$. The connection then follows from the relationship between toric intersection theory and the polytope algebra \cite{FultonSturmfels}.

For a finite set $U$ and
$S,T\subseteq U$, the {\em over-intersection} of $S$ and $T$ in $U$ is
\[\OI_{U}(S,T)=|S\cap T|-\max(0,|S|+|T|-|U|),\]
i.e.,~the quantity by which the size of $S\cap T$ exceeds that which is expected for a ``generic'' choice of $S$ and $T$.
Let $\mult_{U}(|S|,k)=\min(|S|,k)-\frac{k}{|U|}|S|$.

\begin{lemma} \label{l:weights} We have the following identities for $\gamma_k$ (as elements of $A^1(M))$:
\begin{enumerate}
    \item $\gamma_k=\sum_{S\subset E} \OI_{E}(S,T)x_S$ for any set $T$ with $|T|=n+1-k$, and
    \item $\gamma_k=\sum_{S\subset E} \mult_{E}(|S|,k) x_S$.
\end{enumerate}
\end{lemma}

\begin{proof}
For the first identity, take $w=e_T$. Observe that that $e_S\cdot w=|S\cap T|$ and 
\[\min_{|U|=n-k+1}(e_S\cdot e_U)=\max(0,|S|+(n-k+1)-n).\]
Consequently, 
\[-\sum_S \varphi(e_S)x_S=\OI_{E}(S,T)x_S.\]

For the second identity, take $w=\frac{n+1-k}{n+1}e_{E}$, and note that $e_S\cdot w=\frac{n+1-k}{n+1}|S|$. 
\end{proof}

Because the classes $\gamma_k$ arise from the normal fan to a polytope, they are convex. Alternatively one can verify that they are submodular (see \cite[p.~397]{AHK}), e.g.~for the $\OI$-description:
\begin{align*}
    \OI_{E}(\varnothing,T)&=0,\\
    \OI_{E}(E,T)&=0, \text{ and}\\
    \OI_{E}(S_1,T)+\OI(S_2,T)&\geq \OI(S_1\cap S_2,T)+\OI(S_1\cup S_2,T).
\end{align*}

\begin{rem}
  Observe that for $k=1$, we may pick $T=E\setminus\{i\}$ for any $i\in E$. Then $\gamma_1=\sum_{F\not\ni i}x_F.$
  For $k=n$, we can pick $T=\{i\}$ and see $\gamma_n=\sum_{F\ni i}x_F.$
  Thus $\gamma_1$ and $\gamma_n$ are the classes called $\beta$ and $\alpha$, respectively, in \cite{AHK}.
  Consequently, the coefficient of the reduced characteristic polynomial $\mu^i=\deg(\alpha^i\beta^{r-i})$ is a matroidal mixed Eulerian number. We will provide an alternative proof of this fact in Proposition~\ref{p:charpoly}.
\end{rem}

We will find the following expression for the hypersimplex classes helpful:
\begin{lemma} \label{l:flatsizesum} We have the following identity in $A^1(M)$:
\[\gamma_k=(n+1-k)\gamma_n-\sum_{\substack{F\\|F|\geq k+1}} (|F|-k)x_F.\]
\end{lemma}

\begin{proof}
    Fix $T\subset E$ with $|T|=n+1-k$. First observe
    \begin{align*}
        (n+1-k)\gamma_n&=\sum_{i\in T} \left(\sum_{F\ni i} x_F\right)\\
        &= \sum_{F} |F\cap T| x_F.
    \end{align*}
    Now, note
    \begin{align*}
        \gamma_k&=\sum_F \OI_{E}(F,T)x_F\\
        &=\sum_F \left(|F\cap T|-\max(0,|F|+|T|-(n+1))\right)x_F\\
        &=(n+1-k)\gamma_n-\sum_{\substack{F\\|F|\geq k+1}} (|F|-k)x_F. \qedhere
    \end{align*} 
\end{proof}

By convexity of the $\gamma_k$'s and the Hodge theory for $A^*(M)$ \cite[Lemma 9.6]{AHK}, the matroidal mixed Eulerian numbers satisfy a log-concavity property:
\begin{theorem} \label{t:logconcave} Let $c_1,\dots,c_n$ be nonnegative integers with $c_1+\dots+c_n=r-2$. Let $1\leq i,j\leq n$. Then,
\[\deg(\gamma_1^{c_1}\dots\gamma_n^{c_n}\gamma_i^2)\deg(\gamma_1^{c_1}\dots\gamma_n^{c_n}\gamma_j^2)\leq \deg(\gamma_1^{c_1}\dots\gamma_n^{c_n}\gamma_i\gamma_j)^2.\]
\end{theorem}

For a flag of flats 
\[\varnothing\subsetneq F_1\subsetneq F_2\subsetneq\dots\subsetneq F_c\subsetneq E,\]
by inducting on \cite[Prop~2.25]{BHMPW:semismall}, there is an isomorphism $\psi_{\cF}$,
\[
A^*(M)/\ann(x_{F_1}x_{F_2}\dots x_{F_c})\cong
A^*(M^{F_1})\otimes A^*(M_{F_1}^{F_2})\otimes\dots \otimes A^*(M_{F_{c-a}}^{F_c})\otimes A^*(M_{F_c})
\]
induced by multiplication by $x_{F_1}x_{F_2}\dots x_{F_c}$.
Here, $M_{F_j}^{F_{j+1}}$ is a matroid on $F_{j+1}\setminus F_j$ with lattice of flats $[F_j,F_{j+1}]$.
Moreover, if one equips $A^*(M)/\ann(x_{F_1}x_{F_2}\dots x_{F_c})$ with the degree map
\[\deg\colon A^{r-c}(M)/\ann(x_{F_1}x_{F_2}\dots x_{F_c})\to\Z,\ y\mapsto \deg_M(yx_{F_1}x_{F_2}\dots x_{F_c}),\]
and equips $A^*(M^{F_1})\otimes A^*(M_{F_1}^{F_2})\otimes\dots \otimes A^*(M_{F_{c-a}}^{F_c})\otimes A^*(M_{F_c})$
with 
\[\deg_{M^{F_1}}\otimes\deg_{M_{F_1}^{F_2}}\otimes \dots \otimes \deg_{M_{F_c}},\]
then $\psi_{\cF}$ commutes with degree maps.

We now compute the images $\psi_{\cF}(\gamma_k)$ following \cite[Lemma 6.5]{BST}.

\begin{lemma} \label{l:productwithflag}
The image of $\gamma_k$ under $\psi_{\cF}$ is 
\[\begin{cases}
0 &\text{if $k=|F_j|$ for some $j$}\\
1^{\otimes j}\otimes \gamma_{k-|F_j|}\otimes 1^{\otimes(c-j)}&\text{if $|F_j|<k<|F_{j+1}|$ for some $j$}.
\end{cases}
\]
\end{lemma}

\begin{proof}
There is a unique $j$ such that $|F_j|\leq k<|F_{j+1}|$. Pick $T'\subset F_{j+1}\setminus F_j$ of size $|F_{j+1}|- k$, and set $T=(E\setminus F_{j+1})\cup T'$. Then,
  \begin{align*}
      x_{F_1}x_{F_2}\dots x_{F_c}\gamma_k
      &=  x_{F_1}x_{F_2}\dots x_{F_c}\left(\sum_F \OI_{E}(F,T)x_F\right)\\
      &=  x_{F_1}x_{F_2}\dots x_{F_c}\left(\sum_{F\colon F_j\subseteq F\subset F_{j+1}} \OI_{E}(F,T)x_F\right)\\
      &=  x_{F_1}x_{F_2}\dots x_{F_c}\left(\sum_{F\colon F_j\subseteq F\subset F_{j+1}} \OI_{F_{j+1}\setminus F_j}(F\setminus F_j,T')x_F\right)
  \end{align*}
  where the second equality is a consequence of our choice of $T$.
  The conclusion follows from noting that the sum is a description of $\gamma_{k-|F_j|}$ in $A^*(M_{F_j}^{F_{j+1}})$.
\end{proof}

The following is immediate from the above Lemma and Poincar\'{e} duality of $A^*(M)$.

\begin{corollary} \label{c:annihilateflat}
For any flat $F$ with $|F|=k$, $x_F\gamma_k=0$.     
\end{corollary}

Write $\zeta=\sum_{i=1}^n \gamma_i$. The volume of this class, i.e.,~$\on{PVol}(M)=\deg_M(\zeta^{r})$ is called the {\em (standard) permutohedral volume} and is described in terms of the Dilworth truncation of $M$ in \cite[Theorem~7.1.6]{Eur:thesis}. Because the Minkowski sum of the hypersimplexes $\sum_{i=1}^n \Delta(n+1,i)$ is the standard permutohedron, we have $\on{PVol}(U_{n+1,n+1}) = n! (n+1)^{n-1}$, where $(n+1)^{n-1}$ is the volume of this permutohedron. The following is an immediate consequence of the multinomial expansion of $\zeta^r$:

\begin{lemma} \label{l:pvol}
We have
\[\sum_{(c_1,\dots,c_n)} \frac{r!}{c_1!\dots c_n!}A_{c_1,\dots,c_n}(M)=\deg_M(\zeta^r)\]
where the sum is over nonnegative $c_1,\dots,c_n$ with $c_1+\dots+c_n=r$.
\end{lemma}


\section{Relations} \label{s:relations}

The relations satisfied by the matroidal mixed Eulerian numbers are complicated but simplify significantly when one considers some special cases.

\begin{definition}
    The {\em support} of a vector $(c_1,\dots,c_n)$ of nonnegative integers is the set
    \[\Supp(c_1,\dots,c_n)\coloneqq \{i\mid c_i\neq 0\}.\]
\end{definition}

\begin{definition}
    A support set is {\em contiguous} if there exist positive integers $a$ and $b$ such that $\Supp(c)=\{k\mid a\leq k\leq b\}$.
    
    A support set is {\em flatly contiguous} (with respect to the matroid $M$) if there exist positive integers $a$ and $b$ such that $\Supp(c)\subseteq \{k\mid a\leq k\leq b\}$, and
    for any flat $F$ with $a\leq|F|\leq b$, then $|F|\in\Supp(c)$.
\end{definition}

To write relations, it will be helpful to describe matroidal mixed Eulerian numbers a bit  differently. Let $s\leq r$ be a nonnegative integer, and let $v=(v_1,\dots,v_{r-s})\in \N^{r-s}$ be a vector of positive integers.
Write 
\[C_{v,s}(M)=\deg_M(\gamma_{v_1}\gamma_{v_2}\dots\gamma_{v_{r-s}}\gamma_n^s).\]
Then $A_{c_1,\dots,c_n}(M)=C_{v,0}(M)$ for
$v=1^{c_1}2^{c_2}\dots (n-1)^{c_{n-1}}n^{c_n}$,
where $i^{c_i}$ means $i$ appears in $c_i$ consecutive components.
Observe that $C_{v,s}(M)=C_{v,0}(\operatorname{Tr}^s(M))$
where $\operatorname{Tr}^s(M)$ denotes the $s$-fold truncation of $M$.
We may write $\gamma_v$ or $\gamma_{v_1}\gamma_{v_2}\dots\gamma_{v_{r-s}}$.
For a positive integer $k$, write
\[v-k\mathbf{1}=(v_1-k,\dots,v_{r-s}-k).\]
The vector $v$ is {\em sorted} if $v_1\leq v_2\leq \dots\leq v_{r-s}$. Write $R(v,k)\in\N^{r-s-1}$ for the vector obtained by removing the $k$th component of $v$. 

We can define the support of $v$ to be 
\[\Supp(v)=\{v_1,v_2,\dots,v_{r-s}\}.\]
 We say $v$ is {\em (flatly) contiguous} if its support is (flatly) contiguous, in which case, we also say the matroidal mixed Eulerian number $C_{v,s}(M)$ is (flatly) contiguous. Observe that even then, $C_{v,s}(M)=\deg_M(\gamma_v\gamma_n^s)$ is the degree of a product that may not have contiguous support.

Write $f_j=1^{j}0^{r-s-j}$.
It is easily seen that if $v$ is contiguous and sorted, then $R(v,k)+f_{k-1}$ is contiguous. We extend the definition of $C_{v,s}(M)$ to vectors of integers with the convention that $C_{v,s}(M)=0$ if any of the components of $v$ are non-positive.

\begin{lemma} \label{l:contiguousvanishing}
    Let $F$ be a flat of $M$, and let $(v_1,\dots,v_{r-1})\in\N^{r-1}$ be flatly contiguous. 
    Then,
    \[\deg_M(x_F\gamma_v)=0\]
    unless $\rk(F)=1$ or $\rk(F)=r$.
\end{lemma}

\begin{proof}
  Pick $a, b$ as in the definition of contiguity. If $|F|\in \Supp(v)$, then $\deg_M(x_F\gamma_v)=0$ by Corollary \ref{c:annihilateflat}. Hence $|F|<a$ or $|F|>b$. Consider the case $|F|<a$. Now,
  \[\deg_M(x_F\gamma_v)=\deg_{M^F}(1)\deg_{M_F}(\gamma_{v-|F|\mathbf{1}}).\]
  This quantity is $0$ unless $M^F$ is of rank $1$ which occurs only if $\rk(F)=1$.
  The argument for $|F|>b$ is analogous.
\end{proof}

\subsection{Eulerian Relation}

The classical Eulerian recurrence
\[ A(n,k-1) = (n-k+1)A(n-1,k-2) + kA(n-1,k-1) \]
can be rewritten as
\[ \deg_{U_{n+1,n+1}}(\gamma_k^n) = (n-k+1) \deg_{U_{n,n}}(\gamma_{k-1}^{n-1}) + k \deg_{U_{n,n}}(\gamma_k^{n-1}). \]
This recurrence generalizes to contiguous matroidal mixed Eulerian numbers.

\begin{prop} \label{p:eulerianrelation}
  Let $v=(v_1,\dots,v_r)\in\N^r$ be a sorted flatly contiguous vector. Let $j\in\{1,\dots,r\}$ be chosen such that the positive integer $v_j$ occurs at least twice among the components of $v$.
  Let $T\subset E$ be a subset of size $(n+1)-v_j$. Then
  \begin{align*}
  C_{v,0}(M)&=\sum_{\substack{F\\\rk(F)=1}} \OI_{E}(F,T) C_{R(v,j)-\mathbf{1},0}(M_F)\\
  &+\sum_{\substack{F\\\rk(F)=r}} \OI_{E}(F,T)C_{R(v,j),0}(M^F)
   \end{align*}
   Also,
\begin{align*}
  C_{v,0}(M)&=\mult_{E}(|F|,v_j)C_{R(v,j)-\mathbf{1},0}(M_F)\\
  &+\sum_{\substack{F\\\rk(F)=r}}\mult_{E}(|F|,v_j)C_{R(v,j),0}(M^F)       
\end{align*}
\end{prop}

\begin{proof}
  Observe that $R(v,j)$ is still sorted and flatly contiguous. 
  Now, 
  \[\gamma_v=\gamma_{v_j}\gamma_{R(v,j)}=\sum_F \OI_{E}(F,T)x_F\gamma_{R(v,j)}.\]
  By Lemma~\ref{l:contiguousvanishing}, the only nonzero terms in the sum correspond to flats of rank $1$ and rank $r$. The conclusion follows from Lemma~\ref{l:productwithflag}.

  The second formula follows from applying the second description of $\gamma_{v_j}$ in Lemma~\ref{l:weights}.
\end{proof}

We can also make sense of matroidal mixed Eulerian numbers coming from $(v_1,\dots,v_r)$ whose support consists of two flatly contiguous blocks, one containing $1$ and the other containing the size of the largest proper flat.

\begin{lemma} \label{l:contiguousblocks}
    Let $v\in \N^{\ell}$ and $w\in \N^{r-{\ell}}$ be sorted flatly contiguous vectors such that 
    \begin{enumerate}
        \item $\Supp(v)\cap\Supp(w)= \varnothing$,
        \item $1 \in\Supp(v)$, and
        \item if $F$ is a proper flat of maximal size, then $|F|\in \Supp(w)$.
    \end{enumerate}
    Let $w'=(w_2,\dots,w_{r-\ell})$.
    Then,
    \[\deg_M(\gamma_v\gamma_w)=\sum_{F:\on{rk}(F)=\ell+1} \OI_{E}(|F|,T)\deg_{M^F}(\gamma_v)\deg_{M_F}(\gamma_{w'-|F|\mathbf{1}})
    \]
    for any $T\subset E$ with $|T|=n+1-w_1.$
\end{lemma}

\begin{proof}
    We write
    \[\deg_M(\gamma_v\gamma_w)=\sum_F \OI_{E}(|F|,T)\deg(\gamma_v x_F\gamma_{w'}).\]
    The only terms that contribute must satisfy $\max(\Supp(v))<|F|<\min(\Supp(w'))$
    by Lemma~\ref{l:productwithflag}. Therefore, the sum equals
    \[\sum_F \OI_{E}(|F|,T)\deg_{M^F}(\gamma_v)\deg_{M_F}(\gamma_{w'-|F|\mathbf{1}})\]
    which, by dimension considerations, only has contributions from $F$ with $\rk(F)=\ell+1$.
\end{proof}

\subsection{Deletion/Contraction Relations}

We study how the hypersimplex classes behave under $\theta_i$ to prove a deletion/contraction relation.
\begin{lemma} \label{l:twists}
    For $1\leq \ell\leq n$, we have the identities
    \[  \gamma_\ell=\theta_i(\gamma_{\ell-1})+\sum_{\substack{S\not\ni i\\|S|\geq \ell}} x_S,\ \ 
        \gamma_\ell=\theta_i(\gamma_\ell)+\sum_{\substack{S\ni i\\|S|\leq \ell}} x_S.\]
\end{lemma}

\begin{proof}
Without loss of generality, we may suppose $i=n$.
Let $S,T\subseteq E\backslash n$. The following are straightforward verifications.
\begin{align*}
  \OI_{E}(S,T)&=
  \begin{cases}
    \OI_{E\backslash n}(S,T) &\text{if }|S|+|T|-|E\backslash n|\leq 0\\
    \OI_{E\backslash n}(S,T)+1 &\text{if }|S|+|T|-|E\backslash n|\geq 1      
  \end{cases}\\
  \OI_{E}(S\cup\{n\},T)&=\OI_{E\backslash n}(S,T)\\
  \OI_{E}(S,T\cup\{n\})&=\OI_{E\backslash n}(S,T)\\
  \OI_{E}(S\cup\{n\},T\cup\{n\})&=
  \begin{cases}
    \OI_{E\backslash n}(S,T)+1 &\text{if }|S|+|T|-|E\backslash n|\leq -1\\
    \OI_{E\backslash n}(S,T)   &\text{if }|S|+|T|-|E\backslash n|\geq 0      
  \end{cases}
\end{align*}
Pick a set $T\subseteq E\backslash n$ with $|T|=n-\ell$. Then,
\begin{align*}
\theta_n(\gamma_{\ell})&=\sum_{S\subseteq E \backslash n} \OI_{E\backslash n}(S,T)x_S+\sum_{S\subseteq E\backslash n}\OI_{E\backslash n}(S,T)x_{S\cup \{n\}}\\
&=\sum_{S\subseteq E\backslash n} \OI_{E}(S,T)x_S-\sum_{\substack{S\subseteq E\backslash n\\|S|\geq\ell+1}}x_S+\sum_{S\subseteq E\backslash n}\OI_{E}(S\cup\{n\},T)x_{S\cup\{n\}}\\
&=\sum_S\OI_{E}(S,T)x_S-\sum_{\substack{S\not\ni n\\|S|\geq \ell+1}}x_S\\
&=\gamma_{\ell+1}-\sum_{\substack{S\not\ni n\\|S|\geq \ell+1}}x_S,
\end{align*}
which is equivalent to the first identity.
Similarly, 
\begin{align*}
\theta_n(\gamma_{\ell})&=\sum_{S\subseteq E\backslash n} \OI_{E\backslash n}(S,T)x_S+\sum_{S\subseteq E\backslash n}\OI_{E\backslash n}(S,T)x_{S\cup \{n\}}\\
&=\sum_{S\subseteq E\backslash n} \OI_{E}(S,T\cup\{n\})x_S+\sum_{S\subseteq E\backslash n}\OI_{E}(S\cup\{n\},T\cup\{n\})x_{S\cup\{n\}}
-\sum_{\substack{S\subseteq E\backslash n\\|S|\leq\ell-1}}x_{S\cup\{n\}}\\
&=\sum_S\OI_{E}(S,T\cup\{n\})x_S-\sum_{\substack{S\ni n\\|S|\leq \ell}}x_S\\
&=\gamma_{\ell}-\sum_{\substack{S\ni n\\|S|\leq \ell}}x_S,
\end{align*}
giving the second identity.
\end{proof}

\begin{lemma} \label{l:comparewiththeta} We have the following identities:
\begin{enumerate}
    \item for a flat $F$ with $|F|\leq \ell$ and $i\in F$,
\[\theta_i(\gamma_{\ell-1})x_F=\gamma_{\ell}x_F;\]
    \item for a flat $F$ with $|F|\geq \ell$ and $i\not\in F$,
\[\theta_i(\gamma_{\ell})x_F=\gamma_{\ell}x_F;\]
\end{enumerate}
\end{lemma}

\begin{proof}
    For the first identity, Lemma \ref{l:twists} gives
    \[\gamma_\ell-\theta_i(\gamma_{\ell-1})=\sum_{\substack{S\not\ni i\\|S|\geq \ell}} x_S\]
    All flats on the right side must be incomparable with $F$. The proof of the second identity is similar.
\end{proof}

The deletion/contraction relation for contiguous matroidal mixed Eulerian numbers is the following:
\begin{prop} \label{p:contiguousdc} Let $M$ be a loopless matroid of rank at least $3$. Let $0\leq s\leq r$, and let $v=(v_1,\dots,v_{r-s})\in \N^{r-s}$ be a contiguous sorted vector. Let $i\in E$, and set $p=|\overline{\{i\}}|$. Suppose $s=0$ or $v_1=1$.
If  $i$ is not a coloop of $M$, then
\[C_{v,s}(M)=
C_{v,s}(M\backslash i)+\sum_{k=1}^{r-s} C_{R(v,k)+f_{k-1}-p\mathbf{1},s}(M_{\overline{\{i\}}}) 
\]
If $i$ is a coloop of $M$, then 
\[C_{v,s}(M)=C_{v,s-1}(M\backslash i)+\sum_{k=1}^{r-s} C_{R(v,k)+f_{k-1}-\mathbf{1},s}(M_{\overline{\{i\}}}).
\]
\end{prop}

\begin{proof}
    We will rewrite $\deg_M(\gamma_v\gamma_n^s)$ applying the second formula in Lemma~\ref{l:twists} to $\gamma_{v_i}$ and the first formula to $\gamma_n^s$. 
    Observe that for $1\leq j\leq r-s$,
    \begin{align*} 
    \deg_M(\theta_i(\gamma_{v_1}\dots\gamma_{v_{j-1}})\gamma_{v_j}\dots\gamma_{v_{r-s}}\gamma_n^s)&=    \deg_M(\theta_i(\gamma_{v_1}\dots\gamma_{v_j})\gamma_{v_{j+1}}\dots\gamma_{v_{r-s}}\gamma_n^s)\\
    &+\sum_{\substack{F\ni i\\|F|\leq v_j}} \deg_M(x_F \theta_i(\gamma_{v_1}\dots\gamma_{v_{j-1}})\gamma_{v_{j+1}}\dots\gamma_{v_{r-s}}\gamma_n^s).
    \end{align*}
    We claim that only the summands with $\rk(F)=1$ contribute. Suppose $\rk(F)\geq 2$. 
    Let $0\leq k\leq j$ be the largest index such that $v_{k}<|F|$. If $k\geq 1$, there is $\ell$ with $k\leq\ell\leq j-1$ with $v_{\ell}=|F|-1$, and $\theta_i(\gamma_{v_\ell})x_F=\gamma_{|F|}x_F=0$ by Corollary~\ref{c:annihilateflat}. Otherwise, $k=0$ and we must be in the case $v_1>1$, so $s=0$. Now,
    \[\deg_M(x_F \theta_i(\gamma_{v_1}\dots\gamma_{v_{j-1}})\gamma_{v_{j+1}}\dots\gamma_{v_r})=
    \deg_M(x_F \gamma_{{v_1}+1}\dots\gamma_{v_{j-1}+1}\gamma_{v_{j+1}}\dots\gamma_{v_r}).
    \]
    Because this product is contiguous, by Lemma~\ref{l:contiguousvanishing}, the degree vanishes unless $\rk(F)=1$ or $\rk(F)=r$. If $\rk(F)=r$, the degree vanishes by Lemma~\ref{l:productwithflag}. Indeed. the degree is equal to 
    \[\deg_{M^F}(1)\deg_{M_F}(\gamma_{v_1+1}\dots\gamma_{v_{j-1}+1}\gamma_{v_{j+1}}\dots\gamma_{v_r})\]
    which vanishes for dimensional reasons, since $M_F$ is a rank $1$ matroid.
    The condition $\rk(F)=1$ and $i\in F$ forces $F=\overline{\{i\}}$. Thus, the sum is $0$ unless $v_1\geq p=|F|$ in which case it equals
    \begin{align*}
    \deg_M(x_F \gamma_{v_1+1}\dots\gamma_{v_{j-1}+1}\gamma_{v_{j+1}}\dots\gamma_{v_{r-s}}\gamma_n^s)
    &= \deg_{M_F}(\gamma_{v_1+1-p}\dots\gamma_{v_{j-1}+1-p}\gamma_{v_{j+1}-p}\dots\gamma_{v_{r-s}-p}\gamma_{n-p}^s)\\
    &= C_{(v_1+1-p,\dots,v_{j-1}+1-p,v_{j+1}-p,\dots,v_{r-s}-p),s}(M_F)
    \end{align*}
    by applying Lemma~\ref{l:comparewiththeta}.
    By combining the above identities for varying $j$, we obtain
    \[\deg_M(\gamma_{v_1}\dots\gamma_{v_{r-s}}\gamma_n^s)=
    \deg_M(\theta_i(\gamma_{v_1}\dots\gamma_{v_{r-s}})\gamma_n^s)+\sum_{k=1}^{r-s}C_{R(v,k)+f_{k-1}-p\mathbf{1},s}(M_F).\]
    If $s=0$, the conclusion follows. Otherwise, consider the case when 
    $i$ is not a coloop. Then $E\setminus i$ is not a flat and we have 
    \[\gamma_n^s=(\theta_i(\gamma_{n-1})+x_{E\setminus i})^s=\theta_i(\gamma_{n-1}^s).\]
    Again, take degrees. On the other hand, if $i$ is a coloop, then $p=1$, and
    \[\gamma_n=\theta_i(\gamma_{n-1})+x_{E\backslash i}\]
    where $E\backslash i$ is a flat of size $n$. We note that
    \begin{enumerate}
        \item $\gamma_n x_{E\backslash i}=0$ by Lemma~\ref{l:productwithflag}, and
        \item \label{i:thetaeq}  $\theta_i(\gamma_{\ell})x_{E\backslash i}=\gamma_{\ell} x_{E\backslash i}$ by Lemma~\ref{l:comparewiththeta}.  
    \end{enumerate}
    Consequently,
    \begin{align*}
    \deg_M(\theta_i(\gamma_v)\gamma_n^s)&=\deg_M(\theta_i(\gamma_v\gamma_{n-1}^s))+\deg_M(\theta_i(\gamma_v\gamma_{n-1}^{s-1})x_{E\backslash i})\\
    &=\deg_{M\backslash i}(\gamma_v\gamma_{n-1}^{s-1})
   \end{align*}
    where we used $\deg_M\circ\theta_i=0$, \eqref{i:thetaeq}, and Lemma~\ref{l:productwithflag}.
\end{proof}

\section{The Characteristic and Tutte Polynomials} \label{s:tutte}

We relate the characteristic and Tutte polynomial (see, for example, \cite{BO:Tutte}) to the matroidal mixed Eulerian numbers, reproving results of \cite{HuhKatz,BST}. 
We begin with the reduced characteristic polynomial,
\[\overline{\chi}_M(\lambda)=\chi_M(\lambda)/(\lambda-1)\]
where $\chi_M(\lambda)$ is the usual characteristic polynomial. 
We express its coefficients as
  \[\overline{\chi}_M(\lambda)=\sum_{k=0}^r(-1)^k\mu^k(M)\lambda^{r-k}.\]
We can specialize the definition of the reduced characteristic polynomial to loopless matroids to obtain the following characterization by deletion/contraction:
\begin{enumerate}
    \item $\overline{\chi}_{U_{1,1}}(q)=1$,
    \item if $i$ is not a coloop of $M$, then
    \[\overline{\chi}_M(\lambda)=
    \begin{cases}
     \overline{\chi}_{M\backslash i}(\lambda)-\overline{\chi}_{M/i}(\lambda) &\text{if $\{i\}$ is a flat}\\
     \overline{\chi}_{M\backslash i}(\lambda) &\text{otherwise},\\
    \end{cases}\]
    \item and if $i$ is a coloop of $M$, then 
    \[\overline{\chi}_M(\lambda)=(\lambda-1)\overline{\chi}_{M\backslash i}(\lambda).\]    
\end{enumerate}

The following was established in \cite{HuhKatz}, and we provide an alternative proof by deletion-contraction here.
\begin{prop} \label{p:charpoly} For an integer $k$ with $0\leq k\leq r$,
\[\mu^k(M)=\deg_M(\gamma_1^k\gamma_n^{r-k}).\]
\end{prop}

\begin{proof}
    For $U_{1,1}$, this is trivial. For the general case,
    by repeatedly applying 
    \[    \gamma_1=\theta_i(\gamma_1)+x_{\{i\}},\ \gamma_1x_{\{i\}}=0\]
    from Lemma~\ref{l:twists} and Corollary~\ref{c:annihilateflat}, we obtain
    \[\gamma_1^k=\theta_i(\gamma_1^k)+x_{\{i\}}\theta_i(\gamma_1^{k-1}).\]
    Similarly, we obtain
    \[\gamma_n^{r-k}=\theta_i(\gamma_{n-1}^{r-k})+x_{E\backslash i}\theta_i(\gamma_{n-1}^{r-k-1}).\]
    Here, we treat $x_{\{i\}}$ and $x_{E\backslash i}$ as zero if the subscript is not a flat.
    Thus,
    \begin{align*}
        \gamma_1^k\gamma_n^{r-k}&=\theta_i(\gamma_1^k\gamma_{n-1}^{r-k})+x_{\{i\}}\theta_i(\gamma_1^{k-1}\gamma_{n-1}^{r-k})
        +x_{E\backslash i}\theta_i(\gamma_1^k\gamma_{n-1}^{r-k-1})\\
        &=\theta_i(\gamma_1^k\gamma_{n-1}^{r-k})+x_{\{i\}}(\gamma_2^{k-1}\gamma_{n}^{r-k})
        +x_{E\backslash i}(\gamma_1^k\gamma_{n-1}^{r-k-1})
    \end{align*}
    where we made use of the incomparability relation in the matroid Chow ring and Lemma~\ref{l:comparewiththeta}.
    We can, thus, rewrite $\deg_M(\gamma_1^k \gamma_n^{r-k})$ as
    \begin{multline*}
    \deg_{M\backslash i}(\gamma_1^k\gamma_{n-1}^{r-k})+\deg_{M_{\{i\}}}(\gamma_1^{k-1}\gamma_{n-1}^{r-k})
        +\deg_{M \backslash i}(\gamma_1^k\gamma_{n-1}^{r-k-1})\\
    =\begin{cases} \deg_{M\backslash i}(\gamma_1^k\gamma_{n-1}^{r-k})+\deg_{M_{\{i\}}}(\gamma_1^{k-1}\gamma_{n-1}^{r-k}) &\text{if $i$ is not a coloop}\\
    \deg_{M_{\{i\}}}(\gamma_1^{k-1}\gamma_{n-1}^{r-k})
        +\deg_{M \backslash i}(\gamma_1^k\gamma_{n-1}^{r-k-1}) 
    &\text{if $i$ is a coloop}    
    \end{cases}.
    \end{multline*}
    In any case, this is equal to $\mu^k(M\backslash i)+\mu^{k-1}(M/i)$ which, in turn, equals $\mu^k(M)$.
\end{proof}

We can specialize the definition of the Tutte polynomial $T_M(x,y)$ to loopless matroids $M$ and obtain the following characterization:
the Tutte polynomial $T_M(x,y)\in\Z[x,y]$ of a loopless matroid $M$ is a polynomial characterized by the following properties:
\begin{enumerate}
    \item $T_{U_{1,1}}(x,y)=x$,
    \item if $i$ is not a coloop of $M$, 
    \[T_M(x,y)=T_{M\setminus i}(x,y)+y^{p-1}T_{M_{\overline{\{i\}}}}(x,y),\]
    where $p=|\overline{\{i\}}|$, and
    \item if $i$ is a coloop of $M$, then 
    \[T_M(x,y)=xT_{M\setminus i}(x,y).\]
\end{enumerate}

The reduced characteristic polynomial is related to the Tutte polynomial by
\[\overline{\chi}_M(\lambda)=(-1)^{r+1}T_M(1-\lambda,0)/(\lambda-1)\]
We have the following, which was first proven as \cite[Theorem~1.5]{BST}.

\begin{prop} \label{p:contiguoustutte}
For $v=(v_1,\dots,v_{r})\in \N^{r}$, let
\[C_{v}(M,y)=\sum_{k=0}^\infty C_{v+k\mathbf{1},0}(M)y^k.\]
  If $v$ is contiguous and sorted with $v_1=1$, then
  \[C_{v}(M,y)=T_M(1,y)C_{v}(U_{r+1,r+1},y)\]
\end{prop}

\begin{proof}
    Before we begin the proof, we record the following observation: if $w=(w_1,\dots,w_{r-s})$ is contiguous and sorted with $w_1\leq 1$, then 
    \[C_{w-\ell\mathbf{1}}(M,y)=y^{\ell}C_{w}(M,y).\]
    We induct on the number of non-coloops in $M$ by deletion/contraction. If there are no non-coloops, the $M=U_{r+1,r+1}$. In that case $T_M(1,y)=1$, and the result is trivial. Otherwise, let $i\in E$ be a non-coloop of $M$.
    Set $F=\overline{\{i\}}$ and $p=|F|$. Then, by applying Proposition~\ref{p:contiguousdc} to $M$,
    \[ C_{v}(M,y) = C_{v}(M\backslash i,y)+\sum_{k=1}^{r} C_{R(v,k)+f_{k-1}-p\mathbf{1}}(M_F,y). \]
    By induction, we have
    \begin{align*}
        C_{v}(M,y)
        &= T_{M\backslash i}(1,y)C_{v}(U_{r+1,r+1},y)+
        \sum_{k=1}^{r} T_{M_F}(1,y)C_{R(v,k)+f_{k-1}-p\mathbf{1}}(U_{r,r},y)\\
        &= T_{M\backslash i}(1,y)C_{v}(U_{r+1,r+1},y)+
        y^{p-1}T_{M_F}(1,y)\sum_{k=1}^{r} C_{R(v,k)+f_{k-1}-\mathbf{1}}(U_{r,r},y).
    \end{align*}
    Here, the second equality follows from observing that
    \[\min(\Supp(R(v,k)+f_{k-1}-\mathbf{1}))\leq 1\]
    because $v$ is contiguous and sorted with $v_1 = 1$. Now, applying Proposition~\ref{p:contiguousdc} to $U_{r+1,r+1}$ yields
    \begin{align*}
        C_{v}(M,y)
        &= \left(T_{M\backslash i}(1,y)+y^{p-1}T_{M_F}(1,y)\right)C_{v}(U_{r+1,r+1},y)\\
        &= T_M(1,y)C_{v}(U_{r+1,r+1},y).        
    \end{align*}
\end{proof}

By combining the above proposition with Theorem~\ref{t:logconcave}, Berget--Sping--Tseng \cite{BST} were able to resolve a conjecture of Dawson \cite{Dawson}.

\begin{corollary} \label{c:tutte}
  We have the following formulas:
\begin{enumerate}
    \item \cite[Corollary~1.6]{BST} for $v=(1,2,\dots,r)$, we have $C_{v}(M,y)=r!T_M(1,y)$;
    \item $\deg_M(\gamma_1\dots\gamma_r)=r!T_M(1,0)$;
    \item for $v$ contiguous and sorted with $v_1=1$, $C_{v}(M,1)=r!T_M(1,1)$; and
    \item the sum of all contiguous sorted matroidal mixed Eulerian numbers $C_{v,0}(M)$ with $v_1=1$ is \[r!2^{r-1}T_M(1,1).\] 
\end{enumerate}
 
\end{corollary}

\begin{proof}
  By \cite[Theorem~16.3]{Postnikov:permutohedra} (or Lemma~\ref{l:lopsided} applied to $U_{r+1,r+1}$), $C_{v}(U_{r+1,r+1},y)=r!$ for $v = (1,2,\ldots,r)$. The second identity follows by substituting $y=0$.

  For $v$ contiguous and sorted with $v_1=1$, 
  \[C_{v}(M,1)=T_M(1,1)C_{v}(U_{r+1,r+1},1)=r!T_M(1,1)\] 
  where the final equality follows from  \cite[Theorem~16.4]{Postnikov:permutohedra}.
  The final formula comes from summing the above over contiguous sorted $v$ with $v_1=1$, noting that such choices of $v$ are in bijective correspondence with compositions of $r$.
\end{proof}

Proposition~\ref{p:contiguoustutte} allows us to write any contiguous matroidal mixed Eulerian numbers as a convolution of Tutte polynomial coefficients with ordinary mixed Eulerian numbers. Let $v = (v_1,\ldots,v_r)$ be contiguous and sorted, and set $v' = v - (v_1-1)\mathbf{1}$. Then $v'$ has first coordinate $1$ and is also sorted and contiguous. The matroidal mixed Eulerian number $C_{v,0}(M)$ is the coefficient of $y^{v_1-1}$ in $C_{v'}(M,y)$. By Proposition~\ref{p:contiguoustutte}, this evaluates to
  \begin{align*}
  	C_{v,0}(M)
        & = [y^{v_1-1}]C_{v'}(M,y)\\
        & = [y^{v_1-1}](T_M(1,y)C_{v'}(U_{r+1,r+1},y)\\
 &= \sum_{j=0}^{v_1-1} \big( [y^j]T_M(1,y) \big) C_{v' + (v_1-1-j)\mathbf{1},0}(U_{r+1,r+1}) \\
	&= \sum_{j=0}^{v_1-1} \big( [y^j]T_M(1,y) \big) C_{v - j\mathbf{1},0}(U_{r+1,r+1}),
  \end{align*}
where $[y^j]$ denotes taking the coefficient of $y^j$. Note that $C_{v - j\mathbf{1},0}(U_{r+1,r+1}) = 0$ if $j < v_r - r$.

\begin{example}[Uniform matroids]
  When $M = U_{r+1,n+1}$ is a uniform matroid, we have
  \[ T_{U_{r+1,n+1}}(1,y) = \sum_{j=0}^{n-r} \binom{n-j}{r} y^j. \]
  Thus, for any contiguous sorted vector $v$,
  \[ C_{v,0}(U_{r+1,n+1}) = \sum_{j=0}^{v_1-1} \binom{n - j}{r} C_{v - j\mathbf{1},0}(U_{r+1,r+1}). \]
  In particular, the pure powers of the $\gamma_k$ evaluate to
  \[ \deg_{U_{r+1,n+1}}(\gamma_k^r) = \sum_{j=0}^{k-1} \binom{n-j}{r} A(r,k-j-1), \]
  where the $A(r,k-j-1)$ are usual Eulerian numbers.
  For $k \geq r$, we get
  \begin{align*}
  	\deg_{U_{r+1,n+1}}(\gamma_k^r) 
		&= \sum_{j=k-r}^{k-1} \binom{n-j}{r} A(r,k-j-1), \\
		&= \sum_{j=0}^{r-1} \binom{n+1-k+j}{r} A(r,j) \\
		&= (n+1-k)^r
  \end{align*}
  by Worpitzky's identity
  \[ x^r = \sum_{j=0}^r \binom{x+j}{r} A(r,j). \]
\end{example}

\begin{example}[Sparse paving matroids]
  Let $M$ be a sparse paving matroid on $E$ of rank $r+1$ with exactly $m$ circuit-hyperplanes. Then $T_M(1,y) = T_{U_{r+1,n+1}}(1,y) - m$. Consequently, for a contiguous sorted vector $v$,
  \[ C_{v,0}(M) = 
  \begin{cases}
  	C_{v}(U_{r+1,n+1}) - m C_{v}(U_{r+1,r+1}) & \text{if } v_r \leq r \\
	C_{v}(U_{r+1,n+1}) & \text{otherwise.} \\
  \end{cases}
  \] 
\end{example}

\section{Postnikov Trees} \label{s:trees}

In \cite[Section 17]{Postnikov:permutohedra}, Postnikov gave a combinatorial interpretation of mixed Eulerian numbers as weighted counts of certain binary trees. We extend this construction to express an arbitrary product of the $\gamma_i$'s in $A^*(M)$ as a weighted sum of monomials $x_{\cF}$, for $\cF$ a flag of flats of $M$.

Let $T$ be a finite binary tree. 
The {\em binary search order} on the vertices of $T$ is the transitive closure of the relations
\begin{enumerate}
    \item $b\in L_a$ implies $b<a$ and
    \item $b\in R_a$ implies $a<b$,
\end{enumerate}
where $L_a$ and $R_a$ denote the left and right branches, respectively, under $a$. 
Let $\desc(a,T)\coloneqq L_a\cup\{a\}\cup R_a$ be the set of all descendants of $a$.

An \emph{increasing labeling} of the vertex set of $T$ is a bijection
\[\sigma\colon V(T) \to \{1,\dots,k\}\]
such that $\sigma(b)\geq \sigma(a)$ whenever $b\in\desc(a,T)$.
An \emph{increasing binary tree} is a pair $(T,\sigma)$ where $T$ is a binary tree with increasing labeling of $\sigma$. It is well-known that increasing binary trees on the vertex set $\{1, \ldots, k\}$ are in bijection with permutations of $\{1, \ldots, k\}$ \cite[Section 1.5]{Stanley}.

Let $\cL(M)$ denote the lattice of flats of $M$. A {\em flat-filling} of $T$ is a function
\[F\colon V(T)\to \cL(M)\setminus\{\varnothing, E\}\]
such that $a<b$ implies $F(a)\subsetneq F(b)$. Consequently, the image of a flat-filling must be a $k$-step flag of flats $\cF(T,F)$.

A {\em flat-filled increasing binary tree} is a triple $(T,\sigma,F)$, where $(T,\sigma)$ is an increasing binary tree and $F\colon V(T)\to \cL(M)$ 
 is a flat-filling. We will give a necessary {\em compatibility condition} for $(T,\sigma,F)$ to contribute a multiple of $x_{\cF(T,F)}$ in a particular monomial expansion of $\gamma_{v_1}\dots\gamma_{v_k}$ where $k=|V(T)|$. This will allow us to construct a flat-filled increasing binary tree vertex-by-vertex in an order determined by $\sigma$. 
For $i \in \{1,\ldots,k\}$, write $T_{\leq i}$ for the subgraph of $T$ induced by the vertex set $\sigma^{-1}(\{1,\ldots,i\})$. We will also write $F$ for the flat-filling on $T_{\leq i}$ given by restricting $F$ from $T$.
Observe that $T_{\leq i}$ is also a binary tree, and the binary search order on $T_{\leq i}$ is the restriction of the binary search order on $T$. For a vertex $b$, let $\ell(b)$ and $r(b)$ denote $b$'s immediate predecessor and successor, respectively, in the binary search order on $T_{\leq \sigma(b)}$.

For a vector $v = (v_1, \ldots, v_k)$, we define a flat-filled increasing binary tree $(T,\sigma,\cF)$ to be {\em $v$-compatible} if for all $b\in V(T)$
\[ |F(\ell(b))| < v_{\sigma(b)} < |F(r(b))|\]
where if $b$ is the minimal (resp.,~maximal) element in the binary search order on $T$, we set $F(\ell(b))=\varnothing$ (resp.,~$F(r(b))=E$).
In light of Lemma~\ref{l:productwithflag} and the discussion above, this will translate to the condition that $\gamma_{v_{\sigma(b)}}$ could have added the flat $F(b)$ to $\cF(T_{\leq \sigma(b)-1},F)$ to create $\cF(T_{\leq \sigma(b)},F)$.

We define a {\em one-vertex extension} of a flat-filled increasing tree $(T,\sigma,F)$ on $k$ vertices to be a flat-filled increasing tree $(T',\sigma',F')$ on $k+1$ vertices such that $T'_{\leq k} = T$, $\sigma'|_T = \sigma$, and $F'|_{V(T)} = F$.
If the new vertex is called $b$, then $F'(\ell(b))\subsetneq F'(b)\subsetneq F'(r(b))$.

We can choose between two possible natural weights for a $v$-compatible flat-filled increasing binary tree $(T,\sigma,F)$:
\[\wt^{\OI}_v(T,\sigma,F) \coloneqq \prod_{b\in V(T)} \OI_{F(r(b))\setminus F(\ell(b))} (F(b) \setminus F(\ell(b)), U_b)\]
where $U_b$ is the set of the largest $|F(r(b))| - v_{\sigma(b)}$ elements of $F(r(b)) \setminus F(\ell(a))$, or
\[\wt^{\mult}_v(T,\sigma,F) \coloneqq \prod_{b\in V(T)} \mult_{F(r(b))\setminus F(\ell(b))} (F(b) \setminus F(\ell(b)), v_{\sigma(b)}-|F(\ell(a))|).\]
We set the weight of the empty binary tree to be $1$.

\begin{figure}
\begin{tikzpicture}[sibling distance = 30mm, level distance = 10mm, node distance=9mm]

\node[draw,circle,inner sep=1pt,minimum size=2em] (root) {\footnotesize $b_2$} [grow=down]
  child {node[draw,circle,inner sep=1pt,minimum size=2em] (leftchild) {\footnotesize $b_1$}}
  child {node[draw,circle,inner sep=1pt,minimum size=2em] (rightchild) {\footnotesize $b_4$}
	child {node[draw,circle,inner sep=1pt,minimum size=2em] (end) {\footnotesize $b_3$}}
    child {edge from parent[draw = none]}
	}
        ;

\node[above left of=root] {\small $1$};
\node[above left of=leftchild] {\small $3$};
\node[above right of=rightchild] {\small $2$};
\node[right of=end] {\small $4$};


\end{tikzpicture}

\caption{An increasing binary tree. The increasing labeling is given by the numbers next to each vertex.}
\label{f:k=4 example}
\end{figure}

\begin{example}
    Let $M = U_{6,10}$ be the uniform matroid of rank $6$ on ground set $\{0,1,\dots,9\}$. Let $v = (2,3,1,4)$ and consider the increasing binary tree $(T,\sigma)$ on vertex set $\{b_1, b_2, b_3, b_4\}$ pictured in Figure \ref{f:k=4 example}. A flat-filling of $T$ is a a flag
    \[ F(b_1) \subsetneq F(b_2) \subsetneq F(b_3) \subsetneq F(b_4) \]
    of non-empty proper flats. Such a flat-filling is $v$-compatible if and only if $|F(b_2)| = 2$ and $|F(b_4)| = 5$.

    In order for a $v$-compatible flat-filling to have nonzero $\OI$-weight, each of the following conditions must be satisfied:
    \begin{itemize}
        \item $F(b_2) \neq \{0,1\}$.
        \item The minimum element of $E \setminus F(b_2)$ is not contained in $F(b_4)$.
        \item $F(b_1)$ consists of the maximum element of $F(b_2)$.
        \item $F(b_3) \setminus F(b_2)$ must contain the maximum element of $F(b_4) \setminus F(b_2)$.
    \end{itemize}
    The resulting $\OI$-weight will then be equal to $2 - |F(b_2) \cap \{0,1\}|$.

    For instance, the flag
    \[ \{5\} \subsetneq \{3,5\} \subsetneq \{3,5,8\} \subsetneq \{1, 2,3,5,8\} \]
    gives a $v$-compatible flat-filling with $\OI$-weight $2$, whereas the flag
    \[ \{5\} \subsetneq \{0,5\}, \subsetneq \{0,3,5,8\} \subsetneq \{0,2,3,5,8\}\]
    yields a $v$-compatible flat-filling with $\OI$-weight $1$.
\end{example}

\begin{theorem} \label{t:postnikov}
    Let $v = (v_1, \ldots, v_k) \in \N^k$. Then we have the following equality in $A^*(M)$:
    \[ \gamma_{v_1} \cdots \gamma_{v_k} = \sum_{(T,\sigma,F)} \wt_v(T,\sigma,F) \, x_{\cF(T,F)}, \]
    (for either set of weights) where the sum is over all $v$-compatible flat-filled increasing binary trees on $k$ vertices.
\end{theorem}

\begin{proof}
    We will state the proof for the weight function $\wt^{\OI}_v$. The proof using $\wt^{\mult}_v$ is identical.
    The proof is by induction on $k$. For $k=0$, it is trivially true.

    We first claim that for any flat-filled increasing binary tree $(T,\sigma,F)$ with $k$ vertices,
    \[\gamma_{v_{k+1}}x_{\cF(T,F)}=\sum_{(T',\sigma',F')} \OI_{F'(r(b))\setminus F'(\ell(b))} (F'(b) \setminus F'(\ell(b)), U_b)x_{\cF(T',F')}
    \]
    where the sum is over one-vertex extensions of $(T,\sigma,F)$ by a vertex $b$ for which $|F'(\ell(b))|<v_{k+1}<|F'(r(b))|$. 
    Write 
    \[\cF(T,F)=\{\varnothing=F_0\subsetneq F_1\subsetneq F_2\subsetneq \dots\subsetneq F_k\subsetneq E\}.\]
    If $v_{k+1}=|F_j|$ for any $j$, then $\gamma_{v_{k+1}}x_{\cF(T',F')}=0$ and there are no one-vertex extensions $(T',\sigma',F')$ with $|F'(\ell(b))|<v_{k+1}<|F'(r(b))|$.
    Otherwise,  if $|F_j|<v_{k+1}<|F_{j+1}|$, by Lemma~\ref{l:productwithflag}, 
    \[\gamma_vx_{\cF(T,F)}=\sum_{G \colon F_j \subsetneq G \subsetneq F_{j+1}}\OI_{F_{j+1}\setminus F_j} (G \setminus F_j, U)x_Gx_{\cF(T',F')}
    \]
    where $U$ denotes the set of the largest $|F_{j+1}|-v_{k+1}$ elements of $F_{j+1}\setminus F_j$.
    To each $G$ occurring in the above sum, we produce a one-vertex extension of $(T,\sigma,F)$ by adjoining a vertex $b$ such that $F^{-1}(F_j)<b<F^{-1}(F_{j+1})$ in the binary search order and setting $\sigma(b)=k+1$, $F(b)=G$. All one-vertex extensions by $b$ for which $|F'(\ell(b))|<v_{k+1}<|F'(r(b))|$ arise in this fashion.

    Now, we give the inductive step. For $v'=(v_1,\dots,v_{k+1})$, set $v=(v_1,\dots,v_k)$. Write
    \[\gamma_{v_1}\dots\gamma_{v_k}=\sum_{(T,\sigma,F)} \wt^{\OI}_v(T,\sigma,F) \, x_{\cF(T,F)}.
    \]
    Multiply both sides by $\gamma_{v_{k+1}}$. Each choice of $(T,\sigma,F)$ contributes a sum over one-vertex extensions $(T',\sigma',F')$ with $|F'(\ell(b))|<v_{k+1}<|F'(r(b))|$ weighted by an over-intersection term. The condition on the size of flats adjacent to $b$ is exactly $v$-compatibility. Because the product of $\wt^{\OI}_v(T,\sigma,F)$ with the over-intersection term is exactly $\wt^{\OI}_{v'}(T',\sigma',F')$, the conclusion follows.
 
\end{proof}

By taking degrees when $k=r$, we obtain the following:
\begin{corollary} \label{c:postnikov}
    Let $v = (v_1, \ldots, v_r) \in \N^r$. Then,
    \[ C_{v,0}(M) = \sum_{(T,\sigma,F)} \wt_v(T,\sigma,F)\]
    (for either set of weights) where the sum is over all $v$-compatible flat-filled increasing binary trees on $k$ vertices.
\end{corollary}

Postnikov trees give a perspective on the special cases that we have been able to treat in this paper. If $v=(v_1,\dots,v_r)$ is a sorted contiguous vector, by Lemma~\ref{l:contiguousvanishing}, the only binary trees appearing in the above expansion have a path as their underlying tree. It would be interesting to relate such expansions to lattice paths where the two choices of step directions correspond to left and right children.

\begin{figure}
\begin{tikzpicture}[sibling distance = 20mm, level distance = 7mm, node distance=9mm]

\begin{scope}[shift={(-3,0)}]

\node[draw,circle,inner sep=1pt,minimum size=2em] (root) {\footnotesize $b_k$} [grow=down]
  child {node {\reflectbox{$\ddots$}}
        	child {node[draw,circle,inner sep=1pt,minimum size=2em] (child) {\footnotesize $b_2$}
        		child {node[draw,circle,inner sep=1pt,minimum size=2em] (end) {\footnotesize $b_1$}}
		child {edge from parent[draw = none]}
		}
        	child {edge from parent[draw = none]}
        	}
  child {edge from parent[draw = none]}
        ;

\node[above left of=root] {\small $1$};
\node[above left of=child] {\small $k-1$};
\node[above left of=end] {\small $k$};

\node at (-1,-4) [] (a) {(a)};

\end{scope}

\begin{scope}[shift={(3,0)}]

\node[draw,circle,inner sep=1pt,minimum size=2em] (root) {\footnotesize $b_k$} [grow=down]
  child {node {\reflectbox{$\ddots$}}
        	child {node[draw,circle,inner sep=1pt,minimum size=2em] (leftchild) {\footnotesize $b_2$}
        		child {node[draw,circle,inner sep=1pt,minimum size=2em] (leftend) {\footnotesize $b_1$}}
		child {edge from parent[draw = none]}
		}
        	child {edge from parent[draw = none]}
        	}
  child {node[draw,circle,inner sep=1pt,minimum size=2em] (rightchild) {\footnotesize $b_{k+1}$}
  	child {edge from parent[draw = none]}
	child {node {$\ddots$}
		child {edge from parent[draw = none]}
		child {node[draw,circle,inner sep=1pt,minimum size=2em] (rightagain) {\footnotesize $b_{r-1}$}
			child {edge from parent[draw = none]}
			child {node[draw,circle,inner sep=1pt,minimum size=2em] (rightend) {\footnotesize $b_{r}$}}
			}
		}
	}
        ;

\node[above left of=root] {\small $1$};
\node[above left of=leftchild] {\small $k-1$};
\node[above left of=leftend] {\small $k$};
\node[above right of=rightchild] {\small $k+1$};
\node[above right of=rightagain] {\small $r-1$};
\node[above right of=rightend] {\small $r$};

\node at (0.5,-4) [] (b) {(b)};

\end{scope}

\end{tikzpicture}

\caption{Two increasing binary trees. The increasing labeling is given by the numbers next to each vertex. The  unique increasing binary tree compatible with $v = 1^k$ is shown in (a), and the unique increasing binary tree compatible with $v = 1^kn^{r-k}$ is shown in (b).}
\label{f:postnikov}
\end{figure}

\begin{example} \label{ex:descending}
Consider the case of $v = (1,\ldots,1) \in \N^k$, corresponding to the product $\gamma_1^k$. Suppose that $(T,\sigma, F)$ is a $v$-compatible flat-filled increasing binary tree. The $v$-compatibility condition
\[ |F(\ell(b))| < 1 < |F(r(b))|\]
can only be satisfied if $b$ is the minimal vertex in the binary search order on $T_{\leq \sigma(b)}$, so that $F(\ell(b)) = \varnothing$. Suppose this holds for all $b \in V(T)$. Then $\sigma^{-1}(1)$ is the root of $T$, and for $i \in \{2,\ldots, k\}$, $\sigma^{-1}(i)$ is the left child of $\sigma^{-1}(i-1)$. That is, $T$ must be the binary tree which is a path with all edges going to the left and $\sigma$ is the unique increasing labeling of $V(T)$. If we let the vertices of $T$ be $b_1, b_2, \ldots, b_k$ so that $b_1 < b_2 < \cdots < b_k$ in the binary search order, then we have $\sigma(b_i) = k+1-i$. The tree $T$ is pictured in Figure~\ref{f:postnikov}(a).

Now, consider a flat filling $F$ of $(T,\sigma)$. In computing the weight $\wt^{\OI}_v(T,\sigma,F)$, the factor corresponding to $b_k$ is $\OI_{E}(F(b_k),U_{b_k})$, where $U_{b_k} = \{1,\ldots,n\}$ consists of the largest $n$ elements of $E$. This over-intersection is $0$ if $0 \in F(b_k)$ and it is $1$ if $0 \notin F(b_k)$. Similarly, for $1 \leq i \leq k-1$, the over-intersection $\OI_{F(b_{i+1})}(F(b_i), U_{b_i})$ is $0$ unless $\min(F(b_{i+1})) \notin F(b_i)$, in which case it is $1$.

Consequently, the flat-fillings $F$ of $(T,\sigma)$ which have a non-zero weight are precisely those for which $\min(F(b_1)) > \min(F(b_2)) > \cdots > \min(F(b_k))$, i.e.,~those for which the flag $\cF(T,F)$ is {\em descending}. For each such $F$, the weight $\wt^{\OI}_v(T,\sigma,F)$ is $1$. Thus, Theorem~\ref{t:postnikov} recovers the expansion of $\gamma_1^k$ given by \cite[Lemma~9.4]{AHK}.
\end{example}

\begin{example}
Similarly, we compute the degree of $\gamma_1^k \gamma_n^{r-k}$ for $0\leq k\leq r$, which is a coefficient of the reduced characteristic polynomial. We take $v = 1^k n^{r-k}$.

We observe that, again, there is a unique increasing binary tree for which $(T,\sigma,F)$ can be $v$-compatible. Indeed, by reasoning identical to that used in Example~\ref{ex:descending}, the vertices $\sigma^{-1}(1), \ldots, \sigma^{-1}(k)$ must form a path from the root $\sigma^{-1}(1)$ to $\sigma^{-1}(k)$, with all edges going to the left. Compatibility with $v$ then requires that $F(r(b)) = E$ for the remaining vertices, and therefore these vertices form a path descending from the root with all edges going to the right. If we label the vertices $b_1,\ldots, b_r$ with $b_1 < \cdots < b_r$ in the binary search order, then we have
\[ \sigma(b_i) =
\begin{cases}
	k+1-i & \text{if } 1 \leq i \leq k \\
	i & \text{if } k+1 \leq i \leq r
\end{cases}
\]
The tree $(T,\sigma)$ is shown in Figure~\ref{f:postnikov}(b). We note that the tree $T$ possesses $\binom{n}{k}$ distinct increasing labelings. 

We now compute the weight of a flat filling $(T,\sigma,F)$. For $1 \leq i \leq k$, the over-intersection factor coming from $b_i$ is equal to that computed in Example~\ref{ex:descending}. Thus, in order to have $\wt^{\OI}_v(T,\sigma,F) \neq 0$, it must be that
\[ \{ \varnothing \subsetneq F(b_1) \subsetneq \cdots \subsetneq F(b_k) \subsetneq E \} \]
is a descending flag which is {\em initial}, meaning that $\rk(F(b_i)) = i$ for all $1 \leq i \leq k$. The remaining over-intersection factors are $\OI_{E \setminus F(b_{i-1})}(F(b_i) \setminus F(b_{i-1}), U_{b_i})$, where $k+1 \leq i \leq r$, $U_{b_i}=\max(E \setminus F_{i-1})$. Since $\rk(F(b_i)) = \rk(F(b_{i-1})) + 1$, this over-intersection will be non-zero (and equal to $1$) precisely when 
\[F(b_i)=\overline{F(b_{i-1}) \cup \max(E \setminus F(b_{i-1})}.\]
Hence, each initial descending $k$-step flag uniquely determines a summand with weight $1$ in the expansion of $\gamma_1^k \gamma_n^{r-k}$. We thus recover \cite[Proposition~9.5]{AHK}, which states that $\deg_M(\gamma_1^k \gamma_n^{r-k})$ is equal to the number of intial descending flags of length $k$.
\end{example}

\section{Localization, Valuativity and Permutations} \label{s:localization}

\subsection{Equivariant Localization}

In this section, we will use equivariant localization on the toric variety $X(\Delta_{E})$ to show that matroidal mixed Eulerian numbers are valuative and to relate them to permutation statistics. Our main tool will be the equivariant lift of the Bergman fan from \cite{BEST}.

\begin{definition}
Let $M$ be a matroid on $E$, and let $w$ be a permutation in the symmetric group on $E$, $S_{E}$. The flag of flats attached to $w$ is the unique complete flag of flats, 
\[\cF_w=\{\varnothing=F_0\subsetneq F_1\subsetneq \dots\subsetneq F_r\subsetneq F_{r+1}=E\}.\]
obtained from ordering the following set of flats by inclusion:
\[\left\{\overline{\{w(0)\}},\overline{\{w(0),w(1)\}},\dots \overline{\{w(0),\dots,w(n)\}}\right\}.\]
Define an increasing sequence of nonnegative integers 
\[K(w)=\{k_{1,w},\dots,k_{r+1,w}\}\subseteq E\] by 
$k_{j,w}=\min(w^{-1}(F_j\setminus F_{j-1})),$
that is, $k_{j,w}$ has the property that
\[\overline{\{w(0),\dots,w(k_{j,w}-1)\}}\neq \overline{\{w(0),\dots,w(k_{j,w})\}}.\]
Hence, $k_{1,w}=0$. We will suppress $w$ in the notation when it's understood. 
\end{definition}

The set $w\left(\{k_{1,w},\dots,k_{r+1,w}\}\right)$ makes up a basis $B_w(M)$, called {\em the lex-minimal basis}. 
Given a flag of flats $\cF$ and $K\subset E$ with $|K|=r+1$, we may write
\begin{align*}
S_M(\cF,K)&=\{w\in S_{E}\mid \cF_w=\cF,\ K(w)=K\}\\
S_M(\cF)&=\{w\in S_{E}\mid \cF_w=\cF\}
\end{align*}
Then $\{S_M(\cF,K)\}_{(\cF,K)}$ partitions $S_{E}$ as does $\{S_M(\cF)\}_{\cF}$.

Observe that $S_M(\cF,K)$ is the set of $w\in S_{E}$ such that
 for all $j$ with $1\leq j\leq r$, $w(k_j),\dots, w(k_{j+1}-1)\in F_j$, and $w(k_j)\in F_j\setminus F_{j-1}$.

We have the following straightforward:
\begin{lemma} \label{l:smcondition}
Let $w\in S_{E}$. Then $w\in S_M(\cF)$ if and only if  
\[\left(\min(w^{-1}(F_j\setminus F_{j-1})\right)_{j=1,\dots,r+1}\]
forms an increasing sequence. 
\end{lemma}

Localization techniques make use of piecewise polynomials \cite{Brion:piecewise} on $\Delta_{E}$ which can be interpreted as classes in the equivariant Chow ring $A^*_T(M)$. By \cite{EG:equivariant}, there is a non-equivariant restriction map $\iota^*\colon A_T^*(X(\Delta_{E})\to A^*(X(\Delta_{E}))$ which can be interpreted as a map
$\on{PP}^*(\Delta_{E})\to A^*(U_{n+1,n+1})$ where $\on{PP}^*(\Delta_{E})$ denotes the piecewise polynomial functions on $\Delta_{E}$. See  \cite{KP:localization} for additional references for equivariant localization on toric varieties.
By composing this homomorphism with the natural surjection $A^*(U_{n+1,n+1})\to A^*(M)$, we may attach elements of the matroid Chow ring to piecewise polynomials on $\Delta_{E}$. For a piecewise-linear function $\varphi\in \on{PP}^1(\Delta_n)$, this is just the assignment
\[\varphi\mapsto -\sum_F \varphi(e_F)x_F.\]
We will introduce some piecewise polynomials $\lambda_0,\dots,\lambda_n$ on $\Delta_{E}$ 

\begin{definition}
  For $0\leq i\leq n$, we define $\lambda_i\colon \R^{E}\to \R$ by
  \[\lambda_i(x_0,\dots,x_n)=\left((i+1)^\text{st}\text{ highest component of } (x_0,\dots,x_n)\right)-x_n.\]
  Because this is invariant under translation by $\mathbf{1}$, it descends to a piecewise-linear function on $\Delta_{E}$.
\end{definition}

Note that on $\sigma_w$, $\lambda_i$ restricts to $x_{w(i)}-x_n.$

\begin{lemma} We have the following non-equivariant restrictions to $A^1(U_{n+1,n+1})$:
\begin{align*}
    \iota^*\lambda_n&=\gamma_n;\\
    \iota^*(\lambda_k-\lambda_n)&=-\sum_{S\colon |S|\geq k+1} x_S;\\
    \iota^*\lambda_{k}&=\gamma_{k}-\gamma_{k+1}.
\end{align*}
\end{lemma}

\begin{proof}
    We observe that 
    \[\lambda_n(e_S)=
    \begin{cases}
      -1&\text{if $n\in S$}\\
      0 &\text{else}.
    \end{cases}\]
    Consequently $\iota^*\lambda_n=\sum_{S\ni n} x_S=\gamma_n$.

    Similarly,
    \[(\lambda_k-\lambda_n)(e_S)=
    \begin{cases}
      1&\text{if $|S|\geq k+1$}\\
      0 &\text{else}.
    \end{cases}\]

    Finally, using Lemma~\ref{l:flatsizesum}, we see
    \begin{align*}
      \gamma_k-\gamma_{k+1}&=(n+1-k)\gamma_n-(n+1-(k+1))\gamma_n\\
      &+
      \sum_{\substack{S\\|S|\geq k+1}} (|S|-k)x_S
      - \sum_{\substack{S\\|S|\geq k+1}} (|S|-(k+1))x_S\\
      &=\gamma_n-\sum_{\substack{S\\|S|\geq k+1}} x_S\\
      &=\gamma_n+(\iota^*\lambda_{k}-\iota^*\lambda_n)\\
      &= \iota^*\lambda_k.\qedhere
    \end{align*} 
\end{proof}
Consequently, we obtain 
\[\gamma_k=\iota^*\lambda_k+\dots+\iota^*\lambda_n.\]

We will now use the following equivariant localization formula which employs an equivariant lift $c_{\topp}(\mathcal{Q}_M)$ of the Bergman class $[\Delta_M]$ :

\begin{theorem} \label{t:localizationformula}
    Let $M$ be a rank $r+1$ matroid on $E$. Let $\varphi_1,\dots,\varphi_r\in\on{PP}^1(\Delta_{E})$ be piecewise-linear functions on $\Delta_{E}$. For $w\in S_{E}$, write 
    \begin{align*}
    e_{\sigma_w}&=(x_{w(0)}-x_{w(1)})^{-1}\dots (x_{w(n-1)}-x_{w(n)})^{-1},\\
    c_{\topp}(\mathcal{Q}_M)_{\sigma_w}&=(-1)^{n-r}\prod_{i\not\in K(w)}(x_{w(i)}-x_n).
    \end{align*}
    Then,
    \begin{equation} \label{eq:degreeformula}
        \deg_M(\iota^*(\varphi_1\dots\varphi_r))=\sum_{w\in S_{E}} (\varphi_1\dots\varphi_rc_{\topp}(\mathcal{Q}_M))_{\sigma_w}e_{\sigma_w}.
    \end{equation}
\end{theorem}

\begin{proof}
  This is very slight modification of \cite[Theorem~7.6]{BEST}) which considers the action of the algebraic torus $T=(\mathbf{G}_m)^{E}$ on the
  permutohedral toric variety $X(\Delta_{E})$ induced through the projection $\R^E\to \R^{E}/\R\mathbf{1}$. There is an equivariant $K$-class $[\mathcal{Q}_M]\in K^0_T(X(\Delta_{E}))$ whose top Chern class (considered as an element of $A_T^*(X(\Delta_{E}))$, described as a piecewise polynomial) that has the following restriction to $\sigma_w$:
  \[ c_{\topp}(\mathcal{Q}_M)_{\sigma_w}=(-1)^{n-r}\prod_{i\not\in K(w)} t_i\]
  where $t_i$ is the character on $T$ corresponding to a coordinate $T_i$ of $(\mathbf{G}_m)^{E}$. 
  By \cite[Theorem~7.6]{BEST}), the non-equivariant restriction of $c_{\topp}(\mathcal{Q}_M)$ to $A^*(X(\Delta_{E}))$ is the Bergman class $[\Delta_M]$, i.e., one has for  all $c\in A^r(X(\Delta_{E}))$,
  \[\deg_{X(\Delta_{E})}(c \cup [\Delta_M])=\deg_M(c).\]

  Let $T'=(\mathbf{G}_m)^{[E]}/\mathbf{G}_m$ be the quotient of $T$ by the diagonal subtorus. The action of $T$ on $X(\Delta_{E})$ factors through $T'.$ Write $Z_0,\dots,Z_{n-1}$ for the coordinates of $T'$ given by $Z_i=T_iT_n^{-1}$.  The quotient homomorphism $T\to T'$ has a splitting given by $T_i\mapsto Z_i$ for $i=0,\dots,n-1$ and $T_n\mapsto 1$.
  Let $e$ denote the trivial group, so the homomorphisms
  \[e\hookrightarrow T'\hookrightarrow T\]
  induce the restriction maps
  \[A_T^*(X(\Delta_{E}))\to A_{T'}^*(X(\Delta_{E}))\to A^*(X(\Delta_{E})).\]
  The image of $c_{\topp}(\mathcal{Q}_M)\in A^{n-r}_T(X(\Delta_{E}))$ under the first map is the class in the statement of this Lemma, and it maps to $[\Delta_M]$ under the second.
\end{proof}

Observe that $c_{\topp}(\mathcal{Q}_M)_{\sigma_w}$ can be written in terms of $(\lambda_i)_{\sigma_w}$:
\[c_{\topp}(\mathcal{Q}_M)_{\sigma_w}=(-1)^{n-r}\prod_{i\not\in K(w)}(\lambda_i)_{\sigma_w}. \]

In the above, the restriction $(\varphi_i)_{\sigma_w}$ is a linear function and thus $(\varphi_1\dots\varphi_rc_r(\mathcal{Q}_M))_{\sigma_w}$ is a polynomial, i.e.,~an element of $\on{Sym}^* N_{\R}^\vee$.
A priori, the right side of the degree formula \eqref{eq:degreeformula} is only a rational function, i.e.,~an element of the field of fractions of $\on{Sym}^*N_{\R}^\vee$. However, it is in fact equal to an integer (under integrality assumptions on the $\varphi$'s) by localization.

\subsection{Valuativity} \label{ss:valuativity}

We begin by reviewing valuativity of matroid invariants \cite{DerksenFink:valuative}. Recall that for a matroid $M$ on $E$, the matroid polytope is 
\[P(M)=\on{Conv}(\{e_B\mid B\text{ is a basis for }M\}).\]
Write $1_{P(M)}\colon \R^{E}\to \R$ for the characteristic function of $P(M)$. Let $\mathcal{M}(E)$ denote the set of all matroids on $E$. A function $\phi\colon \mathcal{M}(E)\to A$ to an abelian group $A$ is {\em valuative} if for any matroids $M_1,\dots,M_{\ell}$ on $E$ and integers $a_1,\dots,a_{\ell}$ for which $\sum_{i=1}^{\ell} a_i1_{P(M_i)}=0$, we have $\sum_{i=1}^{\ell} a_i\phi(M_i)=0$.

\begin{theorem} Let $c_1,\dots c_n$ be nonnegative integers with $c_1+\dots+c_n=r$. 
    The $\Z$-valued invariant 
    \[M\mapsto \deg_M(\gamma_1^{c_1}\dots\gamma_n^{c_n})\]
    is valuative.
\end{theorem}

\begin{proof}
By \cite[Proposition~5.6]{BEST}, $M\mapsto c_{\topp}(\mathcal{Q}_{M})$ is a valuative invariant in $A^*_T(X(\Delta_{E}))$. Consequently, because hypersimplex classes are obtained from piecewise-linear functions, $\deg_M(\gamma_1^{c_1}\dots\gamma_n^{c_n})$ can be computed by  Theorem~\ref{t:localizationformula} and is thus valuative.
\end{proof}

\subsection{Permutation Statistics}
We discuss the relationship between intersection numbers in the matroid Chow ring and permutation statistics generalizing the results in \cite[Section~3]{Postnikov:permutohedra} on the usual mixed Eulerian numbers. It would be very interesting to compare these arguments to the shelling techniques for the order complex of a matroid \cite[Section~7.6]{Bjorner:shellability} in which permutation statistics also enter.

Recall that for a permutation $w\in S_{E}$,
the descent set is 
\[\Des(w)=\{i\mid w(i)>w(i+1)\},
\]
and we have the descent statistic $\des(w)=|\Des(w)|$.

\begin{definition} \label{d:descentset}
Given  integers $c_0,\dots,c_{n}$ with $c_1+\dots+c_{n}=n$, let
\[I_{c_0,\dots,c_n}=\{i\mid c_0+\dots+c_i<i+1\}.\]

Given integers $d_0,\dots,d_{n}$ with $d_0+\dots+d_n=r$, and a set $K\subset E$ with $|K|=r+1$, let 
\[c_i=\begin{cases}
d_i &\text{if }i\in K\\
d_i+1 &\text{else}, 
\end{cases}\]
and set 
\[I'_{(d_0,\dots,d_n),K}=I_{c_0,\dots,c_n}.\]
\end{definition}

We include a slightly different treatment of the proof given in \cite[Proposition~3.5]{Postnikov:permutohedra} to allow us to evaluate degrees by localization:
\begin{lemma} \label{l:localizationevaluation}
    Let $\xi_0,\dots,\xi_{n-1}$ with $\xi_i=x_i-x_n$ be coordinates on $\R^{E}/\R\mathbf{1}$ (with $\xi_n=0$). Let $w\in S_{E}$. Let $c_0,\dots,c_n$ be  integers with $c_0+\dots+c_n=n$. The constant term of the rational function
    \[\frac{\xi_{w(0)}^{c_0}\dots\xi_{w(n)}^{c_n}}{(\xi_{w(0)}-\xi_{w(1)})\dots (\xi_{w(n-1)}-\xi_{w(n)})}\]
    expressed as a power series in $\Q[\xi^{\pm 1}_0]\ps{\xi_0^{-1}\xi_1,\dots,\xi_{n-1}^{-1}\xi_n}/(\xi_{n-1}^{-1}\xi_n)$
    is $(-1)^{\des(w)}$ if $\Des(w)=I_{c_0,\dots,c_n}$ and $0$ otherwise.
\end{lemma}

\begin{proof}
    The ring $\Q[\xi_0^{\pm 1}]\ps{\xi_0^{-1}\xi_1,\dots,\xi_{n-1}^{-1}\xi_n}/(\xi_{n-1}^{-1}\xi_n)$ contains $\xi_i^{-1}\xi_j$ for any $i<j$. We will eliminate powers of $\xi_i$ in the order $\xi_{w(0)},\dots,\xi_{w(n)}$.
    We prove by induction on $i$ that the constant term in the statement above is equal to the constant term of
    \[(-1)^{|\Des(w)\cap \{0,\dots,i-1\}|}\frac{\xi_{w(i)}^{c_0+\dots+c_i-i}\xi_{w(i+1)}^{c_{i+1}}\dots\xi_{w(n)}^{c_n}}{(\xi_{w(i)}-\xi_{w(i+1)})\dots (\xi_{w(n-1)}-\xi_{w(n)})}\]
    if $\Des(w)\cap \{0,\dots,i-1\}=I_{c_0,\dots,c_n}\cap \{0,\dots,i-1\}$ and is $0$ otherwise.

    We explain the first step, i.e.~that from $i=0$ to $i=1$.
    We note that
    \[\xi_{w(0)}^{c_0}(\xi_{w(0)}-\xi_{w(1)})^{-1}=
    \begin{cases}
        \sum_{k=0}^\infty \xi_{w(0)}^{c_0-k-1}\xi_{w(1)}^k &\text{if }w(0)<w(1)\\
        -\sum_{k=0}^\infty \xi_{w(0)}^{c_0+k}\xi_{w(1)}^{-k-1} &\text{if }w(0)>w(1).
    \end{cases}
    \]
    
    If $w(0)<w(1)$, there is no summand without a negative power of $\xi_{w(0)}$ unless $c_0\geq 1$. In that case, one obtains the summand $\xi_{w(0)}^0\xi_{w(1)}^{c_0-1}$. Thus one obtains a contribution exactly when $0\not\in \Des(w)$ and $0\not\in I_{c_0,\dots,c_n}$.

    If $w(0)<w(1)$, there is no summand without a positive power of $\xi_{w(0)}$ unless $c_0\leq 0$. In that case, one
    obtains $-\xi_{w(0)}^0\xi_{w(1)}^{c_0-1}$. This occurs when $0\in \Des(w)$ and $0\in I_{c_0,\dots,c_n}$.

    Thus, if $\Des(w)\cap \{0\}=I_{c_0,\dots,c_n}\cap \{0\}$, we obtain the expression
    \[(-1)^{|\Des(w)\cap \{0\}|}\frac{\xi_{w(1)}^{c_0+c_1-1}\xi_{w(2)}^{c_{2}}\dots\xi_{w(n)}^{c_n}}{(\xi_{w(1)}-\xi_{w(2)})\dots (\xi_{w(n-1)}-\xi_{w(n)})}.\]
    The general step is analogous.
\end{proof}

\begin{theorem} \label{t:descentcounts}
  Let $M$ be a rank $r+1$ matroid on $E$. Let $d_0,\dots,d_{n}$ be nonnegative integers with $d_0+\dots+d_{n}=r$. Then,
  \[\deg_M(\iota^*(\lambda_0^{d_0}\dots\lambda_{n}^{d_n}))=\sum_w (-1)^{n-r+\des(w)}
  \]
  where the sum is taken over $w\in S_{E}$ with $\Des(w)=I'_{(d_0,\dots,d_n,K(w))}$
\end{theorem}

\begin{proof}
  By the above discussion,
  $\deg_M(\iota^*(\lambda_1^{d_1}\dots\lambda_{n}^{d_n}))$ is equal to the constant term of the rational function
   \[\sum_{w} (\lambda_0^{d_0}\dots\lambda_{n}^{d_{n}}c_{\topp}(\mathcal{Q}_M))_{\sigma_w}e_{\sigma_w}
   =(-1)^{n-r}\sum_{w} (\lambda_0^{c_0}\dots\lambda_{n}^{c_{n}})_{\sigma_w}e_{\sigma_w}\]
   where the $c_i$'s are described in Definition~\ref{d:descentset}.
  On $\sigma_{w}$, the piecewise-linear function $\lambda_0^{c_0}\dots\lambda_n^{c_n}$ restricts to 
  $\xi_{w(0)}^{c_0}\dots \xi_{w(n)}^{c_n}$.
  Thus, the contribution from $w$ is 
  \[(-1)^{n-r} \xi_{w(0)}^{c_0}\dots \xi_{w(n)}^{c_n}/(\xi_{w(0)}-\xi_{w(1)})\dots (\xi_{w(n-1)}-\xi_{w(n)}).\]
  The result now follows from Lemma~\ref{l:localizationevaluation}.
\end{proof}

There may be some combinatorial significance in first summing over the flags of flats, writing
  \[\deg_M(\iota^*(\lambda_0^{d_0}\dots\lambda_{n}^{d_n}))=\sum_\cF \sum_w (-1)^{n-r+\des(w)} 
  \]
(where the inner sum is over $w\in S_M(\cF)$ with $\Des(w)=I'_{(d_0,\dots,d_n,K(w))}$,
and using the condition in Lemma~\ref{l:smcondition} to interpret $S_M(\cF)$.

\section{Perfect Matroid Designs} \label{s:perfect}

In this section, we study the matroidal mixed Eulerian numbers of perfect matroid designs, noting that they encompass the remixed Eulerian numbers of Nadeau and Tewari \cite{NT:remixed} when $q$ is a prime power.

\subsection{Background on perfect matroid designs}
\begin{definition}
    A matroid $M$ is said to be {\em a perfect matroid design} if there are positive integers $n_1<n_2<\dots<n_r$ such that for any full flag of flats
    \[\varnothing=F_0 \subsetneq F_1\subsetneq F_2\subsetneq \dots\subsetneq F_r\subsetneq E,\]
    we have $|F_{j}|=n_j$ for $1\leq j\leq r$.
\end{definition}

For convenience, we will suppose that the matroids in this section are simple, so $n_1=1$. We also set $n_{r+1} = n+1$.
Perfect matroid designs are surveyed in \cite{Deza:PMD}. It is observed in \cite[Proposition~2.2.3]{Deza:PMD}, albeit with a misprint (compare with \cite{DezaSinghi}), that for $1\leq i\leq r$, the number of rank $i$ flats contained in a given rank $i+1$ flat in a perfect matroid design is
\[N_i=\prod_{j=0}^{i-1}\frac{n_{i+1}-n_j}{n_i-n_j}.\]
Note that $N_1=n_2$.

\begin{definition}
    The {\em perfect matroidal mixed Eulerian numbers} are defined to be
    \[A_{(c_1,\dots,c_r)_n}(M)=\deg_M(\gamma_{n_1}^{c_1}\dots\gamma_{n_r}^{c_r})\]
    where $(c_1,\dots,c_r)\in\Z_{\geq 0}^r$ with $c_1+\dots+c_r=r$,
\end{definition}

For perfect matroid designs, the matroidal mixed Eulerian  numbers obey a recurrence that allows any perfect matroidal mixed Eulerian number to be expressed in terms of $A_{(1,\dots,1)_n}(M) = \deg_M(\gamma_{n_1} \cdots \gamma_{n_r})$. Before, we establish it, we identify the matroidal mixed Eulerian numbers of rank $3$ matroids.

\begin{lemma} \label{l:rank3uniform}
    Let $M$ be a rank $3$ perfect matroid design on $E$ where $|F_i|=n_i$ for all flats $F_i$ of rank $i$. Then we have the following matroidal mixed Eulerian numbers:
    \begin{align}
      \label{eq:n1n1} \deg_M(\gamma_{n_1}^2)&=\frac{(n+1-n_1)(n+1-n_2)n_1}{n_2},\\
      \label{eq:n1n2} \deg_M(\gamma_{n_1}\gamma_{n_2})&=(n+1-n_1)(n+1-n_2),\\
      \label{eq:n2n2} \deg_M(\gamma_{n_2}^2)&=(n+1-n_2)^2
    \end{align}
\end{lemma}

\begin{proof}
    First,
    \begin{align*}
    \gamma_{n_1}^2&=\sum_G \mult_{E}(|G|,n_1) x_G\gamma_{n_1}\\
    &=\sum_{G\colon \rk(G)=2} \mult_{E}(|G|,n_1) x_G\gamma_{n_1}\\
    &=\sum_{\varnothing\subsetneq F\subsetneq G\subsetneq E}\mult_{G}(|F|,n_1) \mult_{E}(|G|,n_1)x_Fx_G
    \end{align*}
    where the second equality follows from Lemma~\ref{l:productwithflag}.
    Similarly,
    \begin{align*}
      \gamma_{n_1}\gamma_{n_2}&=\sum_{\varnothing\subsetneq F\subsetneq G\subsetneq E} \mult_{E}(|F|,n_1)\mult_{E\setminus F}(|G|-n_1,n_2-n_1)x_Fx_G,\\
      \gamma_{n_2}^2&=\sum_{\varnothing\subsetneq F\subsetneq G\subsetneq E} \mult_{E}(|F|,n_2)\mult_{E
      \setminus F}(|G|-n_1,n_2-n_1)x_Fx_G,\\
      \gamma_n^2&=\sum_{\varnothing\subsetneq F\subsetneq G\subsetneq E} \mult_{E}(|F|,n)\mult_{E\setminus F}(|G|-n_1,n-n_1)x_Fx_G.
    \end{align*}
    Because the class $x_Fx_G$ in $A^2(M)$ and its coefficient is independent of the flag of flats $F\subsetneq G$, each of these sums is equal to the product of multiplicities with the count of complete flags of flats.
    The result now follows by computing the ratios of the multiplicities in $\frac{\gamma_{n_1}^2}{\gamma_n^2}$, $\frac{\gamma_{n_1}\gamma_{n_2}}{\gamma_n^2}$, and $\frac{\gamma_{n_2}^2}{\gamma_n^2}$,
    and noting that $\deg(\gamma_n^2)=1$ (by Proposition~\ref{p:charpoly}). 
\end{proof}

\begin{lemma} \label{l:perfectrecurrence}
  Let $M$ be a perfect matroid design. Then in $A^*(M)$, we have the relation
  \[\gamma_{n_i}^2=\frac{n_i-n_{i-1}}{n_{i+1}-n_{i-1}}\gamma_{n_i}\gamma_{n_{i+1}}
  +\frac{n_{i+1}-n_i}{n_{i+1}-n_{i-1}}\gamma_{n_{i-1}}\gamma_{n_{i}}.
  \]
\end{lemma}

\begin{proof}
  Our proof follows the lines of the special case of $U_{n+1,n+1}$ in \cite[Corollary~7.9]{BST}.
  By Poincar\'{e} duality on $A^*(M)$, it suffices to verify that the identity is true after multiplying by any $x_{\cF}$ where $\cF$ is a flag of flats of length $r-2$. Write the flag $\cF$ as a subset of a full flag of flats
  \[\varnothing\subsetneq F_1\subsetneq F_2\subsetneq \dots\subsetneq F_r\subsetneq E\]
  with two flats removed. Both sides of the identity are $0$ unless the removed flats are of size $n_i,n_{i+1}$ or of size $n_{i-1},n_i$ by Lemma~\ref{l:productwithflag}. Consider first the case of $n_i,n_{i+1}$.
  Write $y$ for the difference of left and right side of the relation. Now, by Lemma~\ref{l:productwithflag},
  \[\deg_M(x_{\cF}y)=\deg_{M_{F_{i-1}}^{F_{i+2}}}(y),\]
  so we may work in $A^*(M_{F_{i-1}}^{F_{i+2}})$, which is itself a perfect matroid design.
  The images of $\gamma_{n_{i-1}},\gamma_{n_i},\gamma_{n_{i+1}}$ in $A^*(M_{F_{i-1}}^{F_{i+2}})$ are $0,\gamma_{n_i-n_{i-1}},\gamma_{n_{i+1}-n_{i-1}}$, respectively. We conclude $\deg_{M_{F_{i-1}}^{F_{i+2}}}(y)=0$ from Lemma~\ref{l:rank3uniform} which yields
  \[\deg(\gamma_{n_i-n_{i-1}}^2)=\frac{n_i-n_{i-1}}{n_{i+1}-n_{i-1}}\deg(\gamma_{n_i-n_{i-1}}\gamma_{n_{i+1}-n_{i-1}})\]
  from \eqref{eq:n1n1} and \eqref{eq:n1n2}.

  For the case of $n_{i-1},n_i$, we make use of
  \[\deg(\gamma_{n_i-n_{i-2}}^2)=\frac{n_{i+1}-n_i}{n_{i+1}-n_{i-1}}\deg(\gamma_{n_{i-1}-n_{i-2}}\gamma_{n_i-n_{i-2}}).
  \]
  from \eqref{eq:n1n2} and \eqref{eq:n2n2}
\end{proof}

We have the immediate relation among matroidal mixed Eulerian numbers:

\begin{corollary} \label{c:perfectrelation} Let $M$ be a perfect matroid design of rank $r+1$. Let $c=(c_1,\dots,c_n)\in\Z_{\geq 0}^r$ satisfy
$c_1+\dots+c_n=r$. If $c_{n_i}\geq 2$, then
  \[(n_{i+1}-n_{i-1})A_c(M)=(n_i-n_{i-1})A_{c-e_{n_i}+e_{n_{i+1}}}(M)+(n_{i+1}-n_i)A_{c-e_{n_i}+e_{n_{i-1}}}(M).\]
\end{corollary}

For perfect matroidal mixed Eulerian numbers, this relation lends itself to ``probabilistic process'' arguments for these numbers as in \cite{NT:remixed}. 

A particularly explicit case of perfect matroidal mixed Eulerian numbers are the lopsided ones. 

\begin{definition}
    A $r$-tuple of nonnegative integers $(c_1,\dots,c_r)$ is {\em lopsided} if for all $j$ with $1\leq j\leq r$, $\sum_{i=1}^j c_i\geq j$.
    The attached matroidal mixed Eulerian number is $A_{(c_1,\dots,c_r)_n}=\deg_M(\gamma_{n_1}^{c_1}\dots\gamma_{n_r}^{c_r})$.
\end{definition}

\begin{lemma} \label{l:lopsided}
   Let $M$ be a perfect matroid design and let 
   \[V_M=\left(\prod_{i=1}^r  N_i 
 \frac{n_{i+1}-n_i}{n_{i+1}}\right).\]
   If $(c_1,\dots,c_r)\in\Z_{\geq 0}^r$ is lopsided, then
 \[\deg_M(\gamma_{n_1}^{c_1}\dots\gamma_{n_r}^{c_r})=V_Mn_1^{c_1}\dots n_r^{c_r}.\]       
\end{lemma}

\begin{proof}
    Let $\ell$ be the largest index for which $c_\ell$ is positive, and 
    write $(d_1,\dots,d_r)=(c_1,\dots,c_r)-e_\ell$.
    We begin by proving
    \[\deg_M(\gamma_{n_1}^{c_1}\dots\gamma_{n_r}^{c_r})=\sum_{F\colon \rk(F)=r} \mult_{E}(n_r,n_\ell)\deg_{M^F}(\gamma_{n_1}^{d_1}\dots\gamma_{n_\ell}^{d_{\ell}}).\]
    Write 
    \[\deg_M(\gamma_{n_1}^{c_1}\dots\gamma_{n_\ell}^{c_{\ell}})
    =\sum_F \mult_{E}(|F|,n_{\ell})\deg(x_F\gamma_{n_1}^{d_1}\dots \gamma_{n_{\ell}}^{d_{\ell}}).
    \]
    By Corollary~\ref{c:annihilateflat}, only summands for which $|F|\notin\Supp(d)$ contribute a nonzero term. We claim that moreover, summands for which $|F|<n_\ell$ do not contribute. 
    Let $F\notin\Supp(d)$ be a flat of rank $j$ for $j<\ell$.
    Indeed,
    \[D\coloneqq \sum_{i=1}^{j-1} d_i =\sum_{i=1}^j d_i\geq j>\rk(M^F)-1.
    \]
    But by Lemma~\ref{l:productwithflag}, the summand corresponding to $F$ is a multiple of $\deg_{M^F}(\gamma_{n_1}^{d_1}\dots\gamma_{\rk(F)-1}^{d_{\rk(F)-1}})$. This must vanish since $A^D(M^F)=0$ by dimension considerations.
    Now the only summands that contribute must have $|F|\geq n_\ell$. Again by Lemma~\ref{l:productwithflag}, the summand corresponding to $F$ is a multiple of $\deg_{M_F}(1)$ which vanishes for degree conditions unless $\rk(F)=r$.
    Hence, the sum is equal to
    \[\sum_{F\colon \rk(F)=r} \mult_{E}(|F|,n_{\ell})\deg(x_F\gamma_{n_1}^{d_1}\dots \gamma_{n_{\ell}}^{d_{\ell}})
    =N_r\frac{(n+1-n_r)}{n+1}n_{\ell}\deg_{M^F}(\gamma_{n_1}^{d_1}\dots \gamma_{n_{\ell}}^{d_{\ell}}).
    \]
    The result follows by induction because the perfect matroidal mixed Eulerian number on the right is lopsided.
\end{proof}

It would be worthwhile to compute the degree of a product of $\lambda_{n_i}$'s by equivariant localization and compare the product with Nadeau--Tewari's $q$-divided symmetrization \cite{NT:remixed}.

\subsection{Remixed Eulerian Numbers}

In \cite{NT:Klyachko,NT:remixed}, Nadeau and Tewari introduced a $q$-deformation of mixed Eulerian numbers that they call remixed Eulerian numbers. In this section, we identify them (up to a power of $q$) with matroidal mixed Eulerian numbers of a projective geometry.

The projective geometry over the finite field $\F_q$ is the $q$-analogue of the Boolean matroid. 
Recall that 
\[(n)_q=\frac{q^n-1}{q-1}=1+q+\dots+q^{n-1},\ (n)_q!=\prod_{i=1}^n (i)_q.\]
The projective geometry $\PG(r,q)$ is the matroid on the ground set $\PP^r(\F_q)$ where the rank $\rk(S)$  of a subset $S$ is $\dim(\Span(S))+1$, where $\dim(\Span(S))$ is the dimension of its projective span. Consequently, the rank $k$ flats of $\PG(r,q)$ are the $(k-1)$-dimensional subspaces, and each has $n_k\coloneqq (k)_q$ elements. Hence, $\PG(r,q)$ is a perfect matroid design of rank $r+1$.

\begin{definition}
  Let $r$ be a positive integer, and let $\cW_r=\{(c_1,\dots,c_r)\mid c_1+\dots+c_r=r\}$. The remixed Eulerian polynomials $A_c(q)\in\C[q]$ for $c\in \cW_r$ are a collection of polynomials characterized by
  \begin{enumerate}
      \item $A_{1,\dots,1}(q)=(r)_q!$ and
      \item if $c_i\geq 2$, then $(q+1)A_c(q)=qA_{c-e_i+e_{i-1}}(q)+A_{c-e_i+e_{i+1}}(q).$
  \end{enumerate}
\end{definition}

\begin{theorem} For a prime power $q$ and $c\in\cW_r$, we have the identity between matroidal mixed Eulerian numbers and remixed Eulerian numbers: 
\[A_{(c)_n}(\PG(q,r))=q^{\binom{r+1}{2}}A_c(q).\]
\end{theorem}

\begin{proof}
  It is well-known that $N_i=(i+1)_q$. Consequently, by Lemma~\ref{l:lopsided},
  \[A_{(1,\dots,1)_q}(\PG(r,q))=q^{\binom{r+1}{2}}(r)_q!\]

  The relation in Corollary~\ref{c:perfectrelation} becomes the relation satisfied by the remixed Eulerian numbers.
\end{proof}

\begin{rem}
In \cite{NT:Klyachko}, Nadeau--Tewari study a $q$-deformed Klyachko algebra $\cK_{r+1}$ generated by indeterminants $u_1,\dots,u_r$ subject to the relations
\[(q+1)u_i^2=u_iu_{i+1}+qu_{i-1}u_{i}\]
where we take $u_0=u_{r+1}=0$. Our arguments show that there is a natural homomorphism $\cK_{r+1}\to A^r(\PG(r,q))$ given by $u_i\mapsto \gamma_{(i)_q}$. The algebra is equipped with a degree map defined by the $q$-divided symmetrization operation which, therefore, coincides (up to a power of $q$) with the degree map on $A^r(\PG(r,q))$.  Nadeau and Tewari \cite[Section~7]{NT:Klyachko} give a geometric description of the $q$-Klyachko algebra in terms of a Deligne--Lusztig variety. It would be interesting to find a geometric interpretation of the above homomorphism by relating the Deligne--Lusztig variety to an iterated blow-up of projective space $\PP^r$ over $\F_q$.
\end{rem}

\bibliographystyle{plain}
\bibliography{references}

\begin{thebibliography}{10}

\bibitem{AHK}
Karim Adiprasito, June Huh, and Eric Katz.
\newblock Hodge theory for combinatorial geometries.
\newblock {\em Ann. of Math. (2)}, 188(2):381--452, 2018.

\bibitem{BEST}
Andrew Berget, Christopher Eur, Hunter Spink, and Dennis Tseng.
\newblock Tautological classes of matroids.
\newblock {\em Invent. Math.}, 233(2):951--1039, 2023.

\bibitem{BST}
Andrew Berget, Hunter Spink, and Dennis Tseng.
\newblock Log-concavity of matroid {$h$}-vectors and mixed {E}ulerian numbers.
\newblock {\em Duke Math. J.}, 172(18):3475--3520, 2023.

\bibitem{Bjorner:shellability}
Anders Bjorner.
\newblock The homology and shellability of matroids and geometric lattices.
\newblock In {\em Matroid applications}, volume~40 of {\em Encyclopedia Math.
  Appl.}, pages 226--283. Cambridge Univ. Press, Cambridge, 1992.

\bibitem{BHMPW:semismall}
Tom Braden, June Huh, Jacob~P. Matherne, Nicholas Proudfoot, and Botong Wang.
\newblock A semi-small decomposition of the {C}how ring of a matroid.
\newblock {\em Adv. Math.}, 409:Paper No. 108646, 49, 2022.

\bibitem{Brion:piecewise}
Michel Brion.
\newblock Piecewise polynomial functions, convex polytopes and enumerative
  geometry.
\newblock In {\em Parameter spaces ({W}arsaw, 1994)}, volume~36 of {\em Banach
  Center Publ.}, pages 25--44. Polish Acad. Sci. Inst. Math., Warsaw, 1996.

\bibitem{BO:Tutte}
Thomas Brylawski and James Oxley.
\newblock The {T}utte polynomial and its applications.
\newblock In {\em Matroid applications}, volume~40 of {\em Encyclopedia Math.
  Appl.}, pages 123--225. Cambridge Univ. Press, Cambridge, 1992.

\bibitem{CoxLittleSchenck}
David~A. Cox, John~B. Little, and Henry~K. Schenck.
\newblock {\em Toric varieties}, volume 124 of {\em Graduate Studies in
  Mathematics}.
\newblock American Mathematical Society, Providence, RI, 2011.

\bibitem{Dawson}
Jeremy~E. Dawson.
\newblock A collection of sets related to the {T}utte polynomial of a matroid.
\newblock In {\em Graph theory, {S}ingapore 1983}, volume 1073 of {\em Lecture
  Notes in Math.}, pages 193--204. Springer, Berlin, 1984.

\bibitem{DerksenFink:valuative}
Harm Derksen and Alex Fink.
\newblock Valuative invariants for polymatroids.
\newblock {\em Adv. Math.}, 225(4):1840--1892, 2010.

\bibitem{Deza:PMD}
M.~Deza.
\newblock {P}erfect matroid designs.
\newblock In {\em Matroid applications}, volume~40 of {\em Encyclopedia Math.
  Appl.}, pages 54--72. Cambridge Univ. Press, Cambridge, 1992.

\bibitem{DezaSinghi}
M.~Deza and N.~M. Singhi.
\newblock Some properties of perfect matroid designs.
\newblock {\em Ann. Discrete Math.}, 6:57--76, 1980.

\bibitem{EG:equivariant}
Dan Edidin and William Graham.
\newblock Equivariant intersection theory.
\newblock {\em Invent. Math.}, 131(3):595--634, 1998.

\bibitem{Eur:thesis}
Christopher Eur.
\newblock {\em The Geometry of Divisors on Matroids}.
\newblock PhD thesis, University of California, Berkeley, 2020.

\bibitem{FY:chow}
Eva~Maria Feichtner and Sergey Yuzvinsky.
\newblock Chow rings of toric varieties defined by atomic lattices.
\newblock {\em Invent. Math.}, 155(3):515--536, 2004.

\bibitem{FultonSturmfels}
William Fulton and Bernd Sturmfels.
\newblock Intersection theory on toric varieties.
\newblock {\em Topology}, 36(2):335--353, 1997.

\bibitem{Horiguchi}
Tatsuya Horiguchi.
\newblock Mixed {E}ulerian numbers and {P}eterson {S}chubert calculus.
\newblock {\em Int. Math. Res. Not. IMRN}, (2):1422--1471, 2024.

\bibitem{Huh:tropicalgeometry}
June Huh.
\newblock Tropical geometry of matroids.
\newblock In {\em Current developments in mathematics 2016}, pages 1--46. Int.
  Press, Somerville, MA, 2018.

\bibitem{HuhKatz}
June Huh and Eric Katz.
\newblock Log-concavity of characteristic polynomials and the {B}ergman fan of
  matroids.
\newblock {\em Math. Ann.}, 354(3):1103--1116, 2012.

\bibitem{KP:localization}
Eric Katz and Sam Payne.
\newblock Piecewise polynomials, {M}inkowski weights, and localization on toric
  varieties.
\newblock {\em Algebra Number Theory}, 2(2):135--155, 2008.

\bibitem{NT:Klyachko}
Philippe Nadeau and Vasu Tewari.
\newblock A $q$-deformation of an algebra of {K}lyachko and {M}acdonald's
  reduced word formula, 2021.

\bibitem{NT:remixed}
Philippe Nadeau and Vasu Tewari.
\newblock Remixed {E}ulerian numbers.
\newblock {\em Forum Math. Sigma}, 11:Paper No. e65, 26, 2023.

\bibitem{Postnikov:permutohedra}
Alexander Postnikov.
\newblock Permutohedra, associahedra, and beyond.
\newblock {\em Int. Math. Res. Not. IMRN}, (6):1026--1106, 2009.

\bibitem{Stanley}
Richard~P. Stanley.
\newblock {\em Enumerative combinatorics. {V}olume 1}, volume~49 of {\em
  Cambridge Studies in Advanced Mathematics}.
\newblock Cambridge University Press, Cambridge, second edition, 2012.

\end{thebibliography}

\end{document}